\RequirePackage{ifpdf,verbatim}
\ifpdf 
\documentclass[pdftex]{mysigma}
\else
\documentclass{sigma}
\fi

\numberwithin{equation}{section} 
\numberwithin{theorem}{section}
\numberwithin{proposition}{section}
\numberwithin{lemma}{section}
\numberwithin{corollary}{section}
\newtheorem{Theorem}{Theorem}[section]
\newtheorem{Corollary}[Theorem]{Corollary}

\newtheorem{Proposition}[Theorem]{Proposition}

\def\1{^{-1}}

\def\a{\mathfrak{a}}
\def\al{\alpha}

\def\be{\beta}
\def\bi{{\bar\imath\,}}
\def\bj{{\bar\jmath\,}}
\def\bk{{\bar k\,}}
\def\bl{{\bar l\,}}
\def\bp{{\bar p\,}}

\def\br{{\bar r\,}}
\def\bs{{\bar s\,}}

\def\CC{{\mathbb{C}}}
\def\CMN{\CC^{\ts M|N}}

\def\de{\delta}
\def\De{\Delta}

\def\End{\operatorname{End}\hskip1pt}
\def\ep{\varepsilon}

\def\g{\glMN\ts[u]}
\def\ge{\geqslant}
\def\glMN{\mathfrak{gl}_{\ts M|N}}
\def\gr{{\rm gr}\ts}
\def\grpr{{\rm gr}^{\,\prime}\ts}

\def\id{{\mathrm{id}}}

\def\lc{{\ts,\hskip1pt\ldots\ts,\ts}}
\def\le{\leqslant}

\def\ns{{\hskip-.5pt}}

\def\om{\omega}
\def\ot{\otimes}

\def\RS{R^{\,\sharp}}

\def\S{\operatorname{S}}
\def\si{\sigma}
\def\str{\operatorname{str}}
\def\Sym{\mathfrak{S}}

\def\TD{\dot{T}}
\def\TP{T^{\,\prime}}
\def\TS{T^{\,\sharp}}
\def\ts{{\hskip.5pt}}

\def\Ua{\operatorname{U}(\a\ts[u])}
\def\UMN{\operatorname{U}(\glMN)}
\def\Ug{\operatorname{U}(\g)}

\def\XP{X^{\,\prime}}

\def\YMN{\operatorname{Y}(\glMN)}

\def\ZZ{\mathbb{Z}}

\begin{document}


\FirstPageHeading

\ShortArticleName{Yangian of the general linear Lie superalgebra}

\ArticleName{Yangian of the general linear Lie superalgebra}

\Author{Maxim NAZAROV}

\AuthorNameForHeading{Maxim Nazarov}

\Address{Department of Mathematics, University of York, 
York YO10 5DD, United Kingdom}

\bigskip

\Abstract{We prove several basic properties of
the Yangian of the Lie superalgebra $\glMN\,$.}


\Classification{16T20; 17B37; 81R50} 

\makeatletter
\def\special{\xdef\@thefnmark{}\@footnotetext}
\makeatother

\special{This paper is a contribution to the Special Issue on 
Representation Theory and Integrable Systems in honor of Vitaly Tarasov 
on the 60th birthday and Alexander Varchenko on the 70th birthday. 
The full collection is available at
\href{http://www.emis.de/journals/SIGMA/Tarasov-Varchenko.html}
{http://www.emis.de/journals/SIGMA/Tarasov-Varchenko.html}}


\section*{Introduction}

Let $\glMN$ be the general linear Lie superalgebra over the complex field
$\CC\,$. We will
assume that at least one of the non-negative integers $M$ and $N$
is not zero. The Yangian of $\glMN$ has been introduced in \cite{N1}
by extending the definition of the Yangian of the general linear Lie
algebra $\mathfrak{gl}_{\ts M}\,$, see for instance \cite{MNO}. 
We will denote this extension by $\YMN\,$. 
It is a deformation of
the universal enveloping algebra $\Ug$
of the polynomial current 
Lie superalgebra $\g$ in the class of Hopf superalgebras.
The definition of 
$\YMN$ is reviewed in our Section \ref{sec:1}. 

In our Section \ref{sec:1} we will define two ascending filtrations
on the associative algebra $\YMN\,$. The
graded algebra associated with the first filtration
is supercommutative. We prove
that its elements corresponding to the defining
generators \eqref{Tijr} of $\YMN$ are free generators
of this supercommutative algebra.
The graded algebra associated with the second ascending
filtration is isomorphic to $\Ug\,$.
We prove this by using the representation theory of $\YMN\,$. 
Our proof follows \cite{N2} where the Yangian of the queer Lie superalgebra 
$\mathfrak{q}_{\ts M}\subset\mathfrak{gl}_{\ts M|M}$ was studied.
The freeness of the supercommutative graded algebra
associated with the first filtration~on~$\YMN$ 
follows from this isomorphism.
Another proof of the freeness property was given in \cite{G2}. 

Two different families of central elements of $\YMN$
have been defined in \cite{N1}. The definition of the first family uses
the Hopf superalgebra structure on $\YMN\,$. 
This definition is reviewed in our Section \ref{sec:2}. 
It was conjectured in \cite{N1} that the first family 
generates the centre of $\YMN\,$.
Shortly after the publication of \cite{N1}
this conjecture was proved by the author.
The method of that proof was then used in \cite{MNO}
where the Yangian of $\mathfrak{gl}_{\ts M}$ was considered.
This method was also used in \cite{G2,N2}.
We include the original 
proof of this conjecture for $\YMN$ in our Section~\ref{sec:2}.  

The second definition extends the notion of a quantum determinant
for the Yangian of $\mathfrak{gl}_{\ts M}\,$, see again \cite{MNO}
and references therein. This definition is reviewed in our Section \ref{sec:3}.
The main result of \cite{N1} was the relation between the two
families of central elements of $\YMN\,$. However only a summary of
the proof of this relation was given in \cite{N1}
while the details were left unpublished.
The main purpose of the present article is to publish the detailed original 
proof of this relation.

Since the Yangian $\YMN$ was introduced in \cite{N1} it has been studied
by several other authors. Here
we do not not aim to review the literature.
Still let us mention the work \cite{G1}
which contains a direct proof of the centrality
of the elements of $\YMN$ from our second family.
Let us also mention the work \cite{T} which
provides a generalization of $\YMN$ to arbitrary parity sequences.


\section{Definition of the Yangian}
\label{sec:1}

Throughout this article we will use the following general conventions.
Let $\mathrm{A}$ and $\mathrm{B}$ be any two associative
$\ZZ_{\ts2}$-graded algebras. Their tensor product
$\mathrm{A}\ot\mathrm{B}$ is also
an associative $\ZZ_2$-graded algebra such that for any homogeneous
elements
$X,X^\prime\in\mathrm{A}$ and $Y,Y^\prime\in\mathrm{B}$
\begin{align}
\label{conv1}
(X\ot Y)\ts (X^\prime\ot Y^\prime)&=X\ts X^\prime\ot Y\ts Y^\prime
\,(-1)^{\ts\deg X^\prime\deg Y},
\\
\label{conv2}
\deg\ts(X\ot Y)&=\deg X+\deg Y\ts.
\end{align}
For any $\ZZ_2$-graded modules
$U$ and $V$ over $\mathrm{A}$ and $\mathrm{B}$ respectively,
the vector space $U\ot V$ is a $\ZZ_2$-graded module over
$\mathrm{A}\ot\mathrm{B}$ such that for any homogeneous elements
$x\in U$ and $y\in V$
\begin{align}
\label{XY}
(X\ot Y)\ts (\ts x\ot y\ts )&=X\ts x\ot Y\ts y
\,(-1)^{\ts\deg x\,\deg Y},
\\
\label{xy}
\deg\ts(\ts x\ot y\ts)&=\deg x+\deg y\,.
\end{align}
A homomorphism $\al:\mathrm{A}\to\mathrm{B}$ is a linear map
such that $\al\ts(X\,X^\prime)=\al\ts(X)\,\al\ts(X^\prime)$ for all 
$X,X^\prime\in\mathrm{A}\,$. But an antihomomorphism
$\be:\mathrm{A}\to\mathrm{B}$ is a linear map
such that for all homogeneous $X,X^\prime\in\mathrm{A}$ 
\begin{equation}
\label{bb}
\be\ts(X\ts X^\prime\ts)=
\be\ts(X^\prime)\,\be\ts(X)\,(-1)^{\ts\deg X\deg X^\prime}.
\end{equation}

If $\mathrm{A}$ is unital,
let $\iota_{\ts h}$ be 
its embedding into the tensor product $\mathrm{A}^{\ns\ot\ts n}$
as the $h\ts$-th tensor factor\ts:
\begin{equation*}
\iota_{\ts h}(X)=1^{\ot\ts (h-1)}\ot X\ot1^{\ot\ts(n-h)}
\quad\text{for}\quad
h=1\lc n\,.
\end{equation*}
Here $n$ can be any positive integer.
We will also use various embeddings of the algebra
$\mathrm{A}^{\ns\ot\ts m}$ into $\mathrm{A}^{\ns\ot\ts n}$
for any $m=1\lc n$.
For any choice of pairwise distinct indices $h_1\lc h_m\in\{\ts1\lc n\ts\}$
and of an element $X\in\mathrm{A}^{\ns\ot\ts m}$ of the form
$X=X^{(1)}\ot\ldots\ot X^{(m)}$ we will denote
\begin{equation*}
X_{\ts h_1\ldots h_m}=\ts
\iota_{\ts h_1}(X^{(1)})\ts\ldots\,\ts\iota_{\ts h_m}(X^{(m)})
\in\mathrm{A}^{\ns\ot\ts n}.
\end{equation*}
We will then extend the notation $X_{\ts h_1\ldots h_m}$ to all elements
$X\in\mathrm{A}^{\ns\ot\ts m}$ by linearity.

Now let the indices $i\,,j$ run through $1\lc M+N\ts$.
We will always write $\bi=0$ if $1\le i\le M$ and 
$\bi=1$ if $M<i\le M+N\ts$. Consider the
$\ZZ_{\ts2}$-graded vector space $\CMN\ts$.
Let $e_i\in\CMN$ be an element of the standard basis.
The $\ZZ_{\ts2}$-grading on $\CMN$ is defined so that $\deg e_i=\bi\,$.
Let $E_{\ts ij}\in\End\CMN$ be the standard matrix unit, so that 
$E_{\ts ij}\,e_k=\de_{\ts jk}\,e_i\,$. The 
associative algebra $\End\CMN$
is $\ZZ_{\ts2}$-graded so that $\deg E_{\ts ij}=\bi+\ts\bj\,$. 

For any $n$ we can identify
the tensor product $(\End\CMN)^{\ot\ts n}$ with the algebra
$\End((\CMN)^{\ot\ts n})$ acting on the vector space $(\CMN)^{\ot\ts n}$ 
by repeatedly using the conventions \eqref{XY} and \eqref{xy}.

Let us introduce the {\it Yangian\/} of the Lie superalgebra
$\glMN$. This is the complex associative
unital $\ZZ_{\ts2}$-graded algebra $\YMN$ with the countable set of generators 
\begin{equation}
\label{Tijr}
T^{\ts(r)}_{ij}
\quad\text{where}\quad
r=1,2,\,\ldots
\quad\text{and}\quad
i\,,j=1\lc M+N\,.
\end{equation}
The $\ZZ_{\ts2}$-grading on the algebra $\YMN$
is determined by setting $\deg T_{ij}^{\ts(r)}=\bi+\bj$ for $r\ge1$.
To write down defining relations for these
generators we will employ the series
\begin{equation}
\label{3.0}
T_{ij}(u)=
\de_{ij}\cdot1+T_{ij}^{\ts(1)}\ts u^{-1}+T_{ij}^{\ts(2)}\ts u^{-2}+\ldots
\end{equation}
in a formal variable $u$ with coefficients from $\YMN\,$. 
Then for all possible indices $i\,,j\,,k\,,l$
\begin{equation}
\label{3.1}
(u-v)\,
[\,T_{ij}(u)\ts,T_{kl}(v)\,]\,
(-1)^{\ts\,\bi\,\bk\,+\,\bi\,\bl\,+\,\bk\,\bl}=
T_{kj}(u)\,T_{il}(v)-T_{kj}(v)\,T_{il}(u)
\end{equation}
where $v$ is another formal variable.
The square brackets here stand for the supercommutator.
Notice that the series denoted by $T_{ij}(u)$
here and in \cite{N1} differ by the scalar factor 
$(-1)^{\ts\,(\,\bi+1\ts)\,\bj}$.

We will also use the following matrix form of the defining relations
\eqref{3.1}. 
Take the element
\begin{equation}
\label{pop}
P\,=\,\sum_{i,j=1}^{M+N}\,E_{ij}\ot E_{ji}\,(-1)^{\,\bj}
\in(\End\CMN)^{\ot\ts2}\,.
\end{equation}
This element acts on the vector space $(\CMN)^{\ts\ot\ts2}$ so that
$
e_i\ot e_{j}\mapsto e_j\ot e_i\,{(-1)}^{\,\bi\ts\bj}\ts.
$
Here we identify the algebra $(\End\CMN)^{\ot\ts2}$ with the algebra 
$\End((\CMN)^{\ts\ot\ts2})$ by using \eqref{XY}.

For any $n$ let $\Sym_n$ be the symmetric group acting 
on the set $\{1\lc n\}$ by permutations. For each $m=1\lc n-1$
denote by $\si_m$ the element of $\Sym_n$ exchanging $m$ and $m+1\,$.
The group $\Sym_n$ also acts on the vector space $(\CMN)^{\ot\ts n}$.
This action is defined by the assignment 
$\si_m\mapsto P_{m,m+1}$ for each $m\,$.
Here we identify the algebra 
$(\End\CMN)^{\ot\ts n}$ with 
$\End((\CMN)^{\ts\ot\ts n})$ via \eqref{XY},\eqref{xy}.

The rational function $R(u)=1-P\,u\1$ with values in the algebra
$(\End\CMN)^{\ot\ts2}$ is called the {\it Yang R-matrix\/}.
It satisfies the {\it Yang-Baxter equation}
in the algebra $(\End\CMN)^{\ot\ts 3}(u,v,w)$
\begin{equation}
\label{3.6}
R_{\ts12}(u-v)\ts\ts R_{\ts13}(u-w)\ts\ts R_{\ts23}(v-w)=
R_{\ts23}(v-w)\ts\ts R_{\ts13}(u-w)\ts\ts R_{\ts12}(u-v)\,.
\end{equation}
Since $P^2=1\,$, we also have the relation
\begin{equation}
\label{3.52}
R(-u)\,R(u)=1-u^{-2}\,.
\end{equation}

Now combine all the series \eqref{3.0} into the single element
\begin{equation}
\label{tu}
T(u)=\sum_{i,j=1}^{M+N}\,E_{\ts ij}\ot T_{ij}(u)
\in(\End\CMN)\ot\YMN\ts[[u^{-1}]]\,.
\end{equation}
For any $n$ and any $p=1\lc n$ we will denote
\begin{equation}
T_p(u)=\iota_p\ot\id\,(\ts T(u))
\,\in\,
(\End\CMN)^{\ot\ts n}\ot\YMN\,[[u\1]]\,.
\label{3.22}
\end{equation}
By using this notation for $n=2$
the relations \eqref{3.1} can be rewritten as
\begin{equation}
\label{3.3}
(R\ts(u-v)\ot1)\,T_1(u)\,T_2(v)=T_2(v)\,T_1(u)\,(R\ts(u-v)\ot1)\,.
\end{equation}
Namely, after multiplying each side of \eqref{3.3} by $u-v$ it becomes
a relation of series in $u\,,v$ with coefficients in 
$(\End\CMN)^{\ot\ts2}\ot\YMN$
equivalent to the collection of all the relations~\eqref{3.1}.

\begin{Proposition}
\label{m}
An antiautomorphism of\/ 
$\YMN$ can be defined by the assignment
\begin{equation}
\label{M}
T_{ij}(u)\mapsto T_{ij}(-u)\,.
\end{equation}
\end{Proposition}

\begin{proof}
Due to the convention \eqref{conv1}, by using
the notation \eqref{3.22} for $n=2$ we get
\begin{align}
\label{TT}
T_1(u)\,T_2(v)&=\sum_{i,j,k,l=1}^{M+N}\,
E_{\ts ij}\ot E_{\ts kl}\ot T_{ij}(u)\,T_{kl}(v)\,
{(-1)}^{\ts(\,\bi+\,\bj\,)(\ts\bk+\,\bl\,)}\,,
\\
\label{MM}
T_2(-v)\,T_1(-u)&=\sum_{i,j,k,l=1}^{M+N}\,
E_{\ts ij}\ot E_{\ts kl}\ot T_{kl}(-v)\,T_{ij}(-u)\,.
\end{align}
By using the convention \eqref{bb},
the antihomomorphism property of \eqref{M} follows from the relation
\begin{equation*}
(R\ts(u-v)\ot1)\,T_2(-v)\,T_1(-u)=
T_1(-u)\,T_2(-v)\,(R\ts(u-v)\ot1)
\end{equation*}
which 
is obtained from \eqref{3.3} by using \eqref{3.52}.
The antihomomorphism \eqref{M} is clearly involutive and
therefore bijective.
\end{proof}

For all indices $i\,,j$ define the series $\TP_{ij}(u)$
by using the element inverse to \eqref{tu} so that
\begin{equation}
\label{tunat}
T(u)^{-1}=\sum_{i,j=1}^{M+N}\,E_{\ts ij}\ot\TP_{ij}(u)\,.
\end{equation}

\begin{Proposition}
\label{s}
An antiautomorphism of\/ 
$\YMN$ can be defined by the assignment
\begin{equation}
\label{S}
T_{ij}(u)\mapsto\TP_{ij}(u)\,.
\end{equation}
\end{Proposition}

\begin{proof}
Similarly to \eqref{MM}, by using
the notation \eqref{3.22} for $n=2$ we get
\begin{equation*}
T_2(v)^{-1}\,T_1(u)^{-1}=\sum_{i,j,k,l=1}^{M+N}\,
E_{\ts ij}\ot E_{\ts kl}\ot\TP_{kl}(v)\,\TP_{ij}(u)\,.
\end{equation*}
Comparing this with \eqref{TT}
the antihomomorphism property of \eqref{S} follows from the relation
\begin{equation}
\label{3.33}
(R\ts(u-v)\ot1)\,T_2(v)^{-1}\,T_1(u)^{-1}=
T_1(u)^{-1}\,T_2(v)^{-1}\,(R\ts(u-v)\ot1)
\end{equation}
which 
is obtained by multiplying both sides of the defining relation \eqref{3.3} 
on the left and right by $T_2(v)^{-1}$ and then by $T_1(u)^{-1}$.
The bijectivity of \eqref{S} follows from Proposition \ref{ss} below. 
\end{proof}

Further, let $\tau$ be the antiautomorphism of $\End\CMN$
defined by the assignment
\begin{equation*}
E_{ij}\mapsto E_{ji}\,{(-1)}^{\,\bi\ts(\,\bj\ts+\ts1)}\,.
\end{equation*}
Then by the definition \eqref{tu} we have
\begin{equation*}
\tau\ts\ot\id\,(\,T(u)\ts)=
\sum_{i,j=1}^{M+N}\,E_{\ts ji}\ot T_{ij}(u)\,
{(-1)}^{\,\bi\ts(\,\bj\ts+\ts1)}=
\sum_{i,j=1}^{M+N}\,E_{\ts ij}\ot T_{ji}(u)\,
{(-1)}^{\,\bj\ts(\,\bi\ts+\ts1)}\,.
\end{equation*}

\begin{Proposition}
An antiautomorphism of\/ 
$\YMN$ can be defined by the assignment
\begin{equation}
\label{T}
T_{ij}(u)\mapsto T_{ji}(u)\,
{(-1)}^{\,\bj\ts(\,\bi\ts+\ts1)}\,.
\end{equation}
\end{Proposition}

\begin{proof}
Observe that $(\tau\ot\tau)(P)=P$ and hence 
\begin{equation}
\label{Rtt}
(\tau\ot\tau)(R(u-v))=R(u-v)\,.
\end{equation}
Therefore the antihomomorphism property of \eqref{T} 
follows from the relation which is obtained by applying 
$\tau\ot\tau\ot\id$ to both sides of\eqref{3.3},
see the proofs of Propositions \ref{m} and \ref{s} above. 
To prove the bijectivity of 
the antihomomorphism \eqref{T}, observe
that its square is given by 
\begin{equation*}
T_{ij}(u)\mapsto T_{ij}(u)\,
{(-1)}^{\,\bi+\ts\bj}\,.
\end{equation*}
In particular, the square is an automorphism of the $\ZZ_2$-graded 
algebra $\YMN\,$.
\end{proof}

Put
$\TS(u)=\tau\ts\ot\id\,(\,T(u)^{-1}\ts)\,$. 
Then by \eqref{tunat} 
\begin{equation}
\label{tsu}
\TS(u)=
\sum_{i,j=1}^{M+N}\,E_{\ts ji}\ot\TP_{ij}(u)\,
{(-1)}^{\,\bi\ts(\,\bj\ts+\ts1)}=
\sum_{i,j=1}^{M+N}\,E_{\ts ij}\ot\TP_{ji}(u)\,
{(-1)}^{\,\bj\ts(\,\bi\ts+\ts1)}\,.
\end{equation}

\begin{Corollary}
\label{ts}
An automorphism of the algebra $\YMN$ can be defined by the assignment
\begin{equation}
\label{tusharp}
T_{ij}(u)\mapsto\TP_{ji}(u)\,
{(-1)}^{\,\bj\ts(\,\bi\ts+\ts1)}\,.
\end{equation}
\end{Corollary}

\begin{proof}
The assignment \eqref{tusharp} can also be obtained by first applying
\eqref{T} to $T_{ij}(u)$ and then applying \eqref{S} to the result.
Hence \eqref{tusharp} defines an automorphism of the algebra 
$\YMN$ as a composition of two antiautomorphisms.
\end{proof}

By the definition \eqref{tsu} the homomorphism property of 
\eqref{tusharp} is equivalent to the relation 
\begin{equation}
\label{3.333}
(R\ts(u-v)\ot1)\,\TS_1(u)\,\TS_2(v)=
\TS_2(v)\,\TS_1(u)(R\ts(u-v)\ot1)
\end{equation}
which can also be obtained by applying
$\tau\ot\tau\ot\id$ to both sides of \eqref{3.33}
and then using 
\eqref{Rtt}.

\begin{Proposition}
\label{mst}
The antiautomorphisms \eqref{M},\eqref{S},\eqref{T} of\/ $\YMN$
pairwise commute.
\end{Proposition}

\begin{proof}
The antiautomorphism \eqref{M} clearly commutes with either of
\eqref{S},\eqref{T}.
To prove the commutativity of the latter two, let us
consider the tensor product of the antiautomorphisms $\tau\1$ and \eqref{T}
of $\End\CMN$ and $\YMN$ respectively. This is 
is an antiautomorphism of the 
algebra $\End\CMN\ot\YMN\,$, and the series \eqref{tu} is invariant
under this antiautomorphism. 
Hence the series \eqref{tunat}
is also invariant under it. 
The latter invariance implies that \eqref{T} maps\!
\begin{equation*}
\TP_{ij}(u)\mapsto\TP_{ji}(u)\,
{(-1)}^{\,\bj\ts(\,\bi\ts+\ts1)}\,.
\end{equation*}
Hence 
\eqref{tusharp} can also be obtained by first applying
\eqref{S} to $T_{ij}(u)$ and then applying \eqref{T} to the result.
Comparing this with the proof of Corollary \ref{ts}
completes the argument.
\end{proof}

Consider the universal enveloping
algebra $\UMN$  of the Lie superalgebra $\glMN\,$. 
To avoid confusion, the element of $\glMN$
corresponding to $E_{\ts ij}\in\End\CMN$ will be denoted by $e_{\ts ij}\,$. 
By definition, the bracket on $\glMN$ is the supercommutator. 
Hence in $\UMN$
\begin{equation}
\label{glMN}
[\,e_{\ts ij}\,,e_{\ts kl}\,]=
\de_{jk}\,e_{\ts il}-
\de_{\ts li}\,e_{\ts kj}\,
{(-1)}^{\ts(\,\bi+\,\bj\,)(\ts\bk+\,\bl\,)}\,.
\end{equation}
By using the defining relations \eqref{3.1} one can demonstrate that
there is a homomorphism
\begin{equation}
\label{3.5}
\YMN\to\UMN:\,T_{ij}(u)\mapsto \de_{ij}-e_{\ts ji}\,
u^{-1}\,{(-1)}^{\,\bj}\ts.
\end{equation}
This homomorphism is surjective.
The relations \eqref{3.1} also imply that there is a homomorphism\!
\begin{equation}
\label{3.51}
\UMN\to\YMN:\,
e_{\ts ji}\mapsto-\,T_{ij}^{\ts(1)}\,{(-1)}^{\,\bj}\ts.
\end{equation}
The composition of homomorphisms \eqref{3.51}
and \eqref{3.5}
is the identity map $\UMN\to\UMN\,$. So
\eqref{3.51} is an embedding of $\ZZ_2$-graded associative unital algebras.
The homomorphism \eqref{3.5} is identical on the subalgebra $\UMN\,$. 
It is called
the {\it evaluation homomorphism\/} for $\YMN\,$. 

There is a natural Hopf algebra structure on $\YMN\,$.
A coassociative comultiplication homomorphism $\De:\YMN\to\YMN\ot\YMN$ 
can be defined by the assignment
\begin{equation}
\label{3.7}
T_{ij}(u)\mapsto\sum_{k=1}^{M+N}
T_{ik}(u)\ot T_{kj}(u)\,
{(-1)}^{\ts(\ts\bi\ts+\ts\bk\ts)(\ts\bj\ts+\ts\bk\ts)}
\end{equation}
where the tensor product is taken over the subalgebra $\CC[[u^{-1}]]$
in $\YMN\ts[[u\1]]\,$.
The counit homomorphism $\ep:\YMN\to\CC$ is defined 
by the assignment $T_{ij}(u)\mapsto\de_{ij}\,$.
The antipodal mapping $\S:\YMN\to\YMN$ is the antiautomorphism \eqref{S}.
Justification of all these
definitions is similar to that in the case $N=0$ 
considered for instance in \cite[Section~1]{MNO}. Here we omit the details.
Note that \eqref{3.51} is an embedding of Hopf algebras as
by the above definitions
\begin{equation*}
\De:\,T_{ij}^{\ts(1)}\mapsto T_{ij}^{\ts(1)}\ot1+1\ot T_{ij}^{\ts(1)}\ts,
\ \quad
\ep:\,T_{ij}^{\ts(1)}\mapsto0\,,
\ \quad
\S:\,T_{ij}^{\ts(1)}\mapsto-\,T_{ij}^{\ts(1)}\ts.
\end{equation*}

There are two natural 
ascending filtrations on
the associative algebra $\YMN\,$. The first~one is defined by 
assigning the degree $r$ to the generator \eqref{Tijr}.
Let $\gr\YMN$ be the corresponding graded algebra.

\vspace{-1pt}
Let us denote by $X_{ij}^{\ts(r)}$ the image of $T_{ij}^{\ts(r)}$
in the degree $r$ component of $\gr\YMN\,$.
Observe that the $\ZZ_2\ts$-grading on the algebra $\YMN$ descends to
$\gr\YMN$ so that the degree of the image is again
$\bi\ts+\ts\bj\,$. It follows from the relations \eqref{3.1}
that these images supercommute.
We shall prove that  
$\gr\YMN$ is a free supercommutative algebra generated by 
these images. 

Now introduce another filtration on the associative algebra $\YMN$
by assigning the degree $r-1$ to the generator \eqref{Tijr}.
Let $\grpr\YMN$ be the corresponding graded algebra.
Consider the latter algebra.

\vspace{-1pt}
Let us denote by $Y_{ij}^{\ts(r)}$ the image of $T_{ij}^{\ts(r)}$
in the degree $r-1$ component of $\grpr\YMN\,$.
The $\ZZ_2\ts$-grading on the algebra 
$\YMN$ descends to
$\grpr\YMN$ so that the degree of the image is
$\bi\ts+\ts\bj\,$.
The graded algebra $\grpr\YMN$ 
inherits from $\YMN$ the Hopf algebra structure too.
Namely, by the above definitions in 
the graded Hopf algebra $\grpr\YMN$ for $r\ge1$ 
\begin{equation}
\label{grprH}
\De:\,Y_{ij}^{\ts(r)}\mapsto Y_{ij}^{\ts(r)}\ot1+1\ot Y_{ij}^{\ts(r)}\ts,
\ \quad
\ep:\,Y_{ij}^{\ts(r)}\mapsto0\,,
\ \quad
\S:\,Y_{ij}^{\ts(r)}\mapsto-\,Y_{ij}^{\ts(r)}\ts.
\end{equation}

Let us now consider the polynomial current Lie superalgebra $\g\,$.
The elements $e_{ji}\,u^{\ts r}$ with $r=0\ts,1\ts,2\ts,\ts\ldots$
and $i\,,j=1\ts\lc M+N$ make a basis of $\g\,$.
The $\ZZ_{\ts2}$-grading on $\g$ is defined by 
$\deg e_{ji}\,u^{\ts r}=\bi+\bj\,$. Take 
the universal enveloping algebra $\Ug\,$.

\begin{Proposition}
\label{P21}
One can define a surjective homomorphism $\Ug\to\grpr\YMN$
of $\ZZ_2$-graded associative algebras by mapping for $r\ge0$
\begin{equation}
\label{UY}
e_{ji}\,u^r\mapsto-\,Y_{ij}^{\ts(r+1)}\,{(-1)}^{\,\bj}\ts.
\end{equation}
\end{Proposition}

\begin{proof}
Due to \eqref{glMN} the supercommutation relations with $r\ts,s\ge0$
\begin{equation}
\label{grel}
[\,e_{\ts ji}\,u^{\ts r},\,e_{\ts lk}\,u^{\ts s}\,]=
\de_{il}\,e_{\ts jk}\,u^{\ts r+s}-
\de_{\ts kj}\,e_{\ts li}\,u^{\ts r+s}\,
{(-1)}^{\ts(\,\bi+\,\bj\,)(\ts\bk+\,\bl\,)}
\end{equation}
define the Lie superalgebra $\g\,$.
Multiplying the relation \eqref{grel} by 
$(-1)^{\ts\bi\ts\bk\ts+\ts\bi\ts\bl\ts+\ts\bk\ts\bl+\bj+\bl}$
and replacing the basis elements of $\g$ there by their images
under the mapping \eqref{UY} we get
\begin{equation*}
[\ts Y_{ij}^{\ts(r+1)},Y_{kl}^{\ts(s+1)}\ts\bigr]\,
(-1)^{\ts\bi\ts\bk\ts+\ts\bi\ts\bl\ts+\ts\bk\ts\bl}=
\de_{kj}\,Y_{il}^{\ts(r+s+1)}-\de_{il}\,Y_{kj}^{\ts(r+s+1)}\,.
\end{equation*}
The latter relation in $\grpr\YMN$ follows from \eqref{3.1},
see for instance \cite[Section 1]{MNO}.
This proves the homomorphism property of the assignment \eqref{UY}.
This homomorphism is clearly surjective 
and preserves the $\ZZ_{\ts2}$-grading.
\end{proof}

It immediately follows from \eqref{grprH} that
\eqref{UY} is a homomorphism of Hopf algebras.
Here we use the standard Hopf algebra structure on $\Ug$
as the universal enveloping algebra of a Lie superalgebra.
We shall demonstrate that the homomorphism \eqref{UY} is also injective.

By comparing \eqref{3.6} with \eqref{3.3}
we obtain that for any $z\in\CC$ the assignment
\begin{equation}
\label{3.66}
\End\CMN\ot\YMN\ts[[u^{-1}]]
\to
(\End\CMN)^{\ot\ts2}\ts[[u^{-1}]]:\, 
T(u)\mapsto R(u-z)
\end{equation}
defines a representation $\YMN\to\End\CMN$. More explicitly, under
\eqref{3.66} for any $r\ge0$
\begin{equation}
\label{2.19}
T_{ij}^{\ts(r+1)}\mapsto
-\,E_{ji}\,z^{\ts r}\,{(-1)}^{\,\bj}\ts.
\end{equation}
Note that this representation of $\YMN$ 
can be also obtained from the standard representation $\UMN\to\End\CMN$
by pulling back through the evaluation homomorphism \eqref{3.5}
and then back through the automorphism of $\YMN$ defined by mapping
$
T_{ij}(u)\mapsto T_{ij}(u-z)\,.
$ 

The comultiplication on $\YMN$
now allows us to define for $n=1\ts,2\ts,\ts\ldots$ a representation
$\YMN\to(\End\CMN)^{\ot\ts n}$ depending on 
$z_1\lc z_n\in\CC\,$. This is the tensor product of
the representations \eqref{3.66} where $z=z_1\lc z_n\,$.  
Due to \eqref{tu} and \eqref{3.7} under this representation\!
\begin{equation}
\label{3.666}
T(u)\mapsto R_{12}(u-z_1)\ts\ldots\ts R_{1,n+1}(u-z_n)\,.
\end{equation}

\begin{Proposition}
\label{P22}
Let the complex parameters\/ $z_1\lc z_n$ and the positive integer $n$ vary.
Then the kernels of the representations \eqref{3.666} of\/ $\YMN$
have the zero intersection.
\end{Proposition}

\begin{proof}
Take any finite linear combination of the products
$T_{i_1j_1}^{\ts(r_1+1)}\ldots\,T_{i_mj_m}^{\ts(r_m+1)}\in\YMN$ with
$
A_{\,i_1j_1\ldots\ts i_mj_m}^{\,\,r_1\ldots\ts r_m}\in\CC
$
being the respective coefficients.
In this linear combination the number $m\ge0$ and
indices $r_1\lc r_m\ge0$ may vary.
Consider the image of this linear combination 
under the representation 
$\YMN\to(\End\CMN)^{\ot\ts n}$ defined by \eqref{3.666}.
This image depends on $z_1\lc z_n$ polynomially. 
Let $A$ be the sum of those terms of
this polynomial which have the maximal total degree in $z_1\lc z_n\,$.
Let $d$ be this degree.

Consider the second of our two ascending filtrations
on the associative algebra $\YMN\,$,
the corresponding graded algebra being $\grpr\YMN\,$.
Equip the tensor product $\YMN^{\ts\ot\ts n}$ with the ascending 
filtration where the degree is the sum of the degrees on the tensor factors. 
Then by the definition \eqref{3.7}
under the comultiplication $\YMN\to\YMN^{\ts\ot\ts n}$ for $r\ge0$
\begin{equation*}
T_{ij}^{\ts(r+1)}\mapsto\,
\sum_{h=1}^n\,
1^{\ot\ts (h-1)}\ot T_{ij}^{\ts(r+1)}\ot1^{\ot\ts(n-h)}
\,+\,
\text{lower degree terms.}
\end{equation*}
Therefore the sum $A\in(\End\CMN)^{\ot\ts n}$ 
coincides with the image of the sum
\begin{equation*}
\sum_{r_1+\ldots+r_m=\ts d}
A_{\,i_1j_1\ldots\ts i_mj_m}^{\,\,r_1\ldots\ts r_m}\
e_{\ts j_1i_1}\ts u^{\ts r_1}\ldots\,e_{\ts j_mi_m}\ts u^{\ts r_m}\,
(-1)^{\,m\ts+\ts\bj_1\ts+\ts\ldots\ts+\ts\bj_m}\in\Ug
\end{equation*}
under the tensor product of the evaluation representations
\begin{equation}
\label{evalrep}
\Ug\to\End\CMN:\,
e_{ji}\,u^{\ts r}
\mapsto 
E_{ji}\,z^{\ts r}
\end{equation}
at the points $z=z_1\lc z_n\,$. Here we used the explicit
description \eqref{2.19}
of the representation $\YMN\to\End\CMN$ corresponding to $z\in\CC\,$.
Due to Proposition \ref{P21} it now 
suffices to show that when the complex parameters $z_1\lc z_n$
and the positive integer $n$ vary, the kernels
of the tensor products
of the evaluation representations
of the algebra $\Ug$ at $z=z_1\lc z_n$
have the zero intersection.
This will also imply that the homomorphism \eqref{UY} is injective.
 
\vbox{
Choose any $\ZZ_2\ts$-homogeneous 
basis $f_1\lc f_{\ts(M+N)^2}$ of $\glMN$
such that the first vector $f_1$ is 
\begin{equation}
\label{esum}
\sum_{i=1}^{M+N}e_{ii}\,.
\end{equation}
The corresponding elements of $\End\CMN$ 
will be denoted by $F_1\lc F_{\ts(M+N)^2}\,$.
Hence $F_1=1\,$. 
The elements $f_{\ts p}\,u^{\ts r}$ 
with $r=0\ts,1\ts,2\ts,\ts\ldots$
and $p=1\ts\lc (M+N)^2$ constitute a basis of $\g\,$.
Choose any total ordering of this basis which ends with the 
infinite sequence
$
\ldots\,,\ts f_1\ts u^{\ts 2},\ts f_1\ts u\,,\ts f_1\,.
$
Take any finite linear combination $L$ of the products
\begin{equation}
\label{2.3.1}
f_{\ts p_1}u^{\ts r_1}\ldots\,f_{\ts p_m}u^{\ts r_m}\in\Ug
\end{equation}
with $L_{\,p_1\ldots\ts p_m}^{\,r_1\ldots\ts r_m}\in\CC$
being the coefficients.
We assume that the factors in the products are
arranged according to our ordering of the basis of $\g\,$.
Due to the supercommutation relations in $\Ug$ 
we assume it without any loss of generality.
The basis elements of $\ZZ_2\ts$-degree $0$
may occur in any product \eqref{2.3.1} with a multiplicity
but those of $\ZZ_2\ts$-degree $1$ may occur at most~once.
}

Let us denote by $\rho_z$ the evaluation representation \eqref{evalrep}.
More generally, denote by $\rho_{\ts z_1\ldots\ts z_n}$ the tensor product 
$\rho_{\ts z_1}\ot\ldots\ot\rho_{\ts z_n}$
pulled back through $n\ts$-fold comultiplication on $\Ug\,$.
Hence $\rho_{\ts z_1\ldots\ts z_n}$ is a homomorphism of 
associative algebras
\begin{equation*}
\Ug\to(\End\CMN)^{\ot\ts n}:\,
f_{\ts p}\,u^{\ts r}\mapsto
F_{\ts p}\,z_1^{\ts r}\ot1^{\ts\ot\ts(n-1)}+\ldots+
1^{\ts\ot\ts(n-1)}\ot F_{\ts p}\,z_n^{\ts r}\,.
\end{equation*}
Suppose that $\rho_{\ts z_1\ldots\ts z_n}(L)=0$ for every $n$
and all $z_1\lc z_n\in\CC\,$.
We need to prove that $L=0\,$.

For each product \eqref{2.3.1} there is a number $a$ such that 
$p_1\lc p_{\ts a}>1$ but $p_{\ts a+1}\ts\lc p_m=1\,$.
This is due to our ordering of the basis of $\g\,$.
The numbers $a$ for different products
\eqref{2.3.1} may differ, and we do not exclude the case $a=0\,$.
Let $h$ be the maximum of the numbers $a\,$. 

Suppose $n\ge h\,$.
Let $\om_{\ts h}$ be the supersymmetrisation map of
the tensor product $\glMN[u]^{\ts\ot\ts h}$ normalised so that 
$\om_{\ts h}^{\ts2}=h\ts!\,\om_{\ts h}\,$.
Let $W$ be the subspace
of $(\End\CMN)^{\ot\ts n}$ spanned by the vectors
$F_{q_{\ts1}}\ot\ldots\ot F_{q_{\ts n}}$
where at least one
of the indices $q_{\ts1}\lc q_{\,h}$ is $1\,$.
If $h=0$ then this subspace is assumed to be zero.
Applying the homomorphism 
$\rho_{\ts z_1\ldots\ts z_n}$ to a product \eqref{2.3.1} 
with $a=h$ gives
\begin{equation}
\label{2.3.2}
(\ts\rho_{\ts z_1}\ot\ldots\ot\rho_{\ts z_{\ts h}}\ts)\,
(\,\om_{\ts h}\ts
(\ts f_{\ts p_1}u^{\ts r_1}
\ot\ldots\ot 
f_{\ts p_{\ts h}}u^{\ts r_h}\ts)\ts)
\ot1^{\ts\ot\ts(n-h)}
\prod_{b=h+1}^m
(\ts z_1^{\ts r_b}+\ldots+z_n^{\ts r_b}\ts)
\end{equation}
modulo the subspace $W$. Applying
$\rho_{\ts z_1\ldots\ts z_n}$ to a product \eqref{2.3.1} 
with $a<h$ gives an element~of~$W$.
But a linear combination of the 
expressions \eqref{2.3.2} belongs to $W$
only if this combination is zero.

For each product \eqref{2.3.1} with $a=h$
there is a number $c\ge h$ such that $r_{h+1}\ge\ldots\ge r_c>0$ 
but $r_{\ts c+1}\ts\lc r_m=0\,$.
This is due to our ordering of the basis of $\g\,$.
Then \eqref{2.3.2} equals 
\begin{equation}
\label{2.3.3}
(\ts\rho_{\ts z_1}\ot\ldots\ot\rho_{\ts z_{\ts h}}\ts)\,
(\,\om_{\ts h}\ts
(\ts f_{\ts p_1}u^{\ts r_1}
\ot\ldots\ot 
f_{\ts p_{\ts h}}u^{\ts r_h}\ts)\ts)
\ot1^{\ts\ot\ts(n-h)}
\prod_{b=h+1}^c
(\ts z_1^{\ts r_b}+\ldots+z_n^{\ts r_b}\ts)
\end{equation}
multiplied by $n^{\ts m-c}$.
Let $g$ be the maximum of the numbers $c$
for all products \eqref{2.3.1} with $a=h\,$.
 
Suppose $n\ge g\,$. Consider the pairs of sequences of indices 
$p_1\lc p_{\ts m}$ and $r_1\lc r_m$ 
showing in $L$ in the products \eqref{2.3.1}
with $a=h\,$. For every such a pair there is
$c\in\{\ts h\lc g\ts\}\,$. 
Then we have
$p_{\ts h+1}\lc p_{\ts m}=1$ and $r_{c+1}\lc r_m=0\,$.
Take all different pairs of sequences $p_1\lc p_{\ts h}$ and $r_1\lc r_c$
arising in this way.
The expressions \eqref{2.3.3} corresponding to the latter pairs
are linearly independent as polynomials in $z_1\lc z_n$ 
with values in $(\End\CMN)^{\ot\ts n}\,$.
This is again due to our ordering of the basis of $\g\,$.
Here we also employ the observation that~in~\eqref{2.3.3} the image of
$\rho_{\ts z_1}\ot\ldots\ot\rho_{\ts z_{\ts h}}$
does not depend on the parameters $z_{\ts h+1}\lc z_n$
while the product over $b=h+1\lc c$ in \eqref{2.3.3}
does depend on these parameters whenever $c>h\,$. 
Therefore if $\rho_{\ts z_1\ldots\ts z_n}(L)=0$ for a certain $n\ge g$
and for all $z_1\lc z_n\in\CC$ then for each of the latter pairs
\begin{equation*}
\sum_{m=c}^\infty\ 
L_{\,p_1\ldots p_h\,1\ldots\ts 1}^{\,r_1\ldots\,r_c\,0\ldots\ts0}\ 
n^{\ts m-c}=0\,.
\end{equation*}
There are exactly $m$ lower indices and also $m$ upper indices
in the coefficient 
$L_{\,p_1\ldots p_h\,1\ldots\ts 1}^{\,r_1\ldots\,r_c\,0\ldots\ts0}$ above.
By letting the number $n$ vary we now prove that all these 
coefficients vanish. 
\end{proof}

In the course of the proof of Proposition \ref{P22} we established that
the homomorphism \eqref{UY} is injective.
Together with Proposition \ref{P21} 
and the observation made just after its proof this~yields 

\begin{Theorem}
\label{T18}
The Hopf algebras\/ $\Ug$ and\/ $\grpr\YMN$
are isomorphic via \eqref{UY}.
\end{Theorem}

Let us now invoke the Poincar\'e\ts-Birkhoff-Witt theorem for the universal
enveloping algebras of Lie superalgebras \cite[Theorem 5.15]{MM}.
By applying this theorem to
the Lie superalgebra $\g$ we now obtain its analogue 
for the Yangian $\YMN\,$.

\begin{Corollary}
The supercommutative algebra\/ $\gr\YMN$ is freely generated 
by $X_{ij}^{\ts(r)}$ of the $\ZZ_2$-degree $\bi+\bj$ 
where $r=1\ts,2\ts,\ts\ldots$ and $i\,,j=1\ts\lc M+N\,$. 
\end{Corollary}


\section{Centre of the Yangian}
\label{sec:2}

Here we will give a description of the centre of the
algebra $\YMN\,$. An element of $\YMN$ is called
\emph {central} if it supercommutes with each element of $\YMN\,$.
However, we will see that the central elements of $\YMN$ have
the $\ZZ_{\ts2}$-degree $0\,$.
Hence they commute with each 
element of $\YMN$ in the usual sense.
Our description comes from computing the square 
of the antipodal mapping $\S$ of $\YMN\,$.
Another description of the centre of $\YMN\,$ will be 
given in the next section. In that section we will also establish a 
correspondence between the two descriptions.

\begin{Proposition}
\label{zu}
There is a formal power series $Z(u)$ in $u^{-1}$ with coefficients in
the centre of\/ $\YMN$ and with 
leading term $1$ such that for all indices $i$ and $j$
\begin{align}
\label{zetav}
\sum_{k=1}^{M+N}
T_{kj}(u+M-N)\,\TP_{ik}(u)&=\delta_{ij}\,Z(u)\,,
\\[4pt]
\label{zetau}
\sum_{k=1}^{M+N}
\TP_{kj}(u)\,T_{ik}(u+M-N)&=\delta_{ij}\,Z(u)\,.
\end{align}
\end{Proposition}

\begin{proof}
By using the definition \eqref{pop} introduce the element of
the algebra $(\End\CMN)^{\ts\ot\ts2}$
\begin{equation}
\label{Q}
Q=(\ts\id\ot\tau\ts)\ts(\ts P\ts)=
\sum_{i,j=1}^{M+N}E_{ij}\ot E_{ij}\,{(-1)}^{\,\bi\ts\bj}\ts.
\end{equation}
The image of the action of $Q$ on 
$(\CMN)^{\ts\ot\ts2}$ is one dimensional and is spanned by the vector
\begin{equation*}
\sum_{i=1}^{M+N}e_{i}\ot e_{i}\,.
\end{equation*}
Here we regard $Q$ as an element of the algebra 
$\End((\CMN)^{\ts\ot\ts2})$ 
by identifying this algebra with $(\End\CMN)^{\ts\ot\ts2}$ via \eqref{XY}.
We also have
$Q^{\ts2}=(M-N)\,Q\,$. By using the latter relation we get
\begin{equation}
\label{rutinv}
((\ts\id\ot\tau\ts)\ts(\ts R(u)\ts))^{-1}=(\ts1-Q\,u\1)^{-1}=
1+Q\,(\ts u-M+N\ts)^{-1}\,.
\end{equation}
The rational function of the variable $u$ 
given by the equalities \eqref{rutinv} will be denoted by $\RS(u)\,$.

Now multiply both sides of the relation 
\eqref{3.3}
by $T_2^{-1}(v)$ on the left and right, and then apply
$\tau$ relative
to the second tensor factor $\End\CMN$ in 
$(\End\CMN)^{\ot2}\ot\YMN\,$.
Multiplying the resulting relation on the left and right by 
$\RS(u-v)\ot 1$ yields
\begin{equation}
\label{trater}
(\RS(u-v)\ot1)\,
T_1(u)\,\TS_2(v)=
\TS_2(v)\ts T_1(u)\,
(\RS(u-v)\ot1)\,.
\end{equation}
Multiplying the latter relation by $u-v-M+N$ and then setting
$u=v+M-N$ we get
\begin{equation}
\label{qresi}
(Q\ot1)\,T_1(v+M-N)\,\TS_2(v)=
\TS_2(v)\,T_1(v+M-N)\,(Q\ot1)\,,
\end{equation}
see \eqref{rutinv}.
Since the image of $Q$ in $(\CMN)^{\ts\ot\ts2}$
is one dimensional, either side of the relation \eqref{qresi}
must be equal to $Q\ot Z(v)$ where $Z(v)$
is a certain power series in $v^{-1}$
with coefficients from $\YMN\,$. 
The equality 
of the left hand side and of the right hand side of \eqref{qresi}
to $Q\ot Z(v)$ 
is equivalent respectively to \eqref{zetav} and \eqref{zetau}. 
We just need to replace the variable $v$ in \eqref{qresi} by $u\,$.

It is immediate from \eqref{3.0} and \eqref{zetav} that
every coefficient of the series $Z(u)$ has
$\ZZ_{\ts2}$-degree $0$ and that its leading
term is $1\,$. 
Let us prove that all these coefficients
are central in $\YMN\,$.

To this end we will work with the elements \eqref{3.22}
where $n=3\,$. By using \eqref{3.3} and \eqref{trater},
\begin{gather}
(\RS_{13}(u-v)\,R_{12}(u-v-M+N)\ot1)\,
T_1(u)\,T_2(v+M-N)\,\TS_3(v)=
\notag
\\
(\RS_{13}(u-v)\ot1)\,
T_2(v+M-N)\,T_1(u)\,\TS_3(v)\,
(R_{12}(u-v-M+N)\ot1)=
\notag
\\
T_2(v+M-N)\,\TS_3(v)\,T_1(u)\,
(\RS_{13}(u-v)\,R_{12}(u-v-M+N)\ot1)\,.
\label{3.666666}
\end{gather}
On the other hand, due to \eqref{3.52} we have the identity in the algebra
$(\End\CMN)^{\ot\ts3}(\ts u\ts)$
\begin{equation*}
R_{13}(-u)\,P_{23}\,R_{12}(u)=(\,1-u^{-2}\,)\,P_{23}\,.
\end{equation*}
By applying to it the antiautomorphism $\tau$
relative to the third tensor factor of $(\End\CMN)^{\ot\ts3}$
and then using the definition \eqref{rutinv} we get
\begin{equation}
\label{QR}
Q_{23}\,\RS_{13}(u+M-N)\,R_{12}(u)=(\,1-u^{-2}\,)\,Q_{23}\,.
\end{equation}
Therefore multiplying the first and third lines of the display
\eqref{3.666666} by $Q_{23}\ot1$ on the left yields
\begin{equation*}
(\ts1-(u-v-M+N)^{-2}\ts)\,
T_1(u)\,(Q_{23}\ot Z(v))\,=\,
(Q_{23}\ot Z(v))\,T_1(u)\,
(\ts1-(u-v-M+N)^{-2}\ts)\,.
\end{equation*}
We just need to replace the variable $u$ in \eqref{QR} by $u-v-M+N$
and use the relation \eqref{qresi}.
The last relation implies that any generator $T_{ij}^{(r)}$
commutes with every coefficient of 
$Z(v)\,$.
\end{proof}

The square $\S^{\ts2}$ of antipodal mapping
is an automorphism of the associative algebra $\YMN\,$.

\begin{Proposition}
\label{ss}
The automorphism\/ $\S^{\ts2}\!$ of $\YMN$ is given by the assignment
\begin{equation*}
T_{ij}(u)\mapsto Z(u)^{-1}\,T_{ij}(u+M-N)\,.
\end{equation*}
\end{Proposition}

\begin{proof}
By the definition \eqref{tunat} for any indices $i\,,j$ 
we have the identity
\begin{equation}
\label{TTP}
\sum_{k=1}^{M+N}\,T_{ik}(u)\,\TP_{kj}(u)\,
{(-1)}^{\ts(\ts\bi\ts+\ts\bk\ts)(\ts\bj\ts+\ts\bk\ts)}
=\delta_{ij}\,.
\end{equation}
Here we use the convention \eqref{conv1}. By using the definition of $\S$
this identity can be written as 
\begin{equation*}
\sum_{k=1}^{M+N}\,T_{ik}(u)\ts\S(\ts T_{kj}(u))\,
{(-1)}^{\ts(\ts\bi\ts+\ts\bk\ts)(\ts\bj\ts+\ts\bk\ts)}
=\delta_{ij}\,.
\end{equation*}
Let us apply the antiautomorphism $\S$ to both sides of
the latter identity. We get 
\begin{equation*}
\sum_{k=1}^{M+N}\,
\S^{\ts2}(\ts T_{kj}(u))\ts\S(\ts T_{ik}(u))=\delta_{ij}
\quad\text{or}\quad
\sum_{k=1}^{M+N}\,
\S^{\ts2}(\ts T_{kj}(u))\,\TP_{ik}(u)=\delta_{ij}\,.
\end{equation*}
By comparing the last displayed identity with \eqref{zetav} we see that
for any indices $k\,,j$
\begin{equation*}
Z(u)\ts\S^{\ts2}(\ts T_{kj}(u))=T_{kj}(u+M-N)\,.
\tag*{\qed}
\end{equation*}
\renewcommand{\qed}{}
\end{proof}

\begin{Corollary}
\label{DZ}
For the formal power series $Z(u)$ in $u^{-1}$ we have
\begin{equation}
\label{DeZu}
\De:\,Z(u)\mapsto Z(u)\ot Z(u)\,,
\ \quad
\ep:\,Z(u)\mapsto1\,,
\ \quad
\S:\,Z(u)\mapsto Z(u)^{-1}\,.
\end{equation}
\end{Corollary}

\begin{proof}
The square of the antipodal mapping is a coalgebra homomorphism. 
Hence the images of
$T_{ij}(u)$ relative to the two compositions
$\De\ts\S^{\ts2}$ and $(\S^{\ts2}\ot\S^{\ts2})\,\De$
are the same. By Proposition~\ref{ss} these images are respectively equal to
\begin{gather*}
\De\,(\ts Z(u)^{-1}\,T_{ij}(u+M-N))=
\De\,(\ts Z(u)^{-1})\,\times
\\
\sum_{k=1}^{M+N}
T_{ik}(u+M-N)\ot T_{kj}(u+M-N)\,
{(-1)}^{\ts(\ts\bi\ts+\ts\bk\ts)(\ts\bj\ts+\ts\bk\ts)}\,,
\end{gather*}
and
\begin{gather*}
\sum_{k=1}^{M+N}
\S^{\ts2}(\ts T_{ik}(u))\ot\S^{\ts2}(\ts T_{kj}(u))\,
{(-1)}^{\ts(\ts\bi\ts+\ts\bk\ts)(\ts\bj\ts+\ts\bk\ts)}=
(\ts Z(u)^{-1}\ot Z(u)^{-1})\,\times
\\
\sum_{k=1}^{M+N}
T_{ik}(u+M-N)\ot T_{kj}(u+M-N)\,
{(-1)}^{\ts(\ts\bi\ts+\ts\bk\ts)(\ts\bj\ts+\ts\bk\ts)}\,.
\end{gather*}
By equating the two images of $T_{ij}(u)$ we obtain that 
\begin{equation*}
\De:\,Z(u)^{-1}\mapsto Z(u)^{-1}\ot Z(u)^{-1}\,.
\end{equation*}
Since the comultiplication $\De$ is an algebra homomorphism,
we get the first statement in \eqref{DeZu}.

The second statement in \eqref{DeZu} immediately follows from 
\eqref{zetav} and from the definition of the counit homomorphism $\ep\,$.
The third statement follows from the first and the second because the
multiplication $\mu:\YMN\ot\YMN\to\YMN$ and the unit mapping
$\de:\CC\to\YMN$ satisfy the identity
$
\mu\,(\S\ot\id)\,\De=
\de\,\ep\,.
$
Indeed, 
by applying to the coefficients of the series 
$Z(u)$ the homomorphisms at the two
sides of this identity we get the equality $\S(Z(u))\,Z(u)=1\,$.
\end{proof}

In Section \ref{sec:3} we will also use the following observation.
For any indices 
$i\,,j$ by 
\eqref{tunat} we have
\begin{equation*}
\sum_{k=1}^{M+N}\,\TP_{ik}(u)\,T_{kj}(u)\,
{(-1)}^{\ts(\ts\bi\ts+\ts\bk\ts)(\ts\bj\ts+\ts\bk\ts)}
=\delta_{ij}
\end{equation*}
similarly to \eqref{TTP}.
The collection of last displayed identities can be written as a 
single relation 
\begin{equation}
\label{QTT}
(Q\ot1)\,\TS_2(u)\,T_1(u)=Q\ot1
\end{equation}
of series in $u$ with coefficients in the algebra 
$(\End\CMN)^{\ts\ot\ts2}\ot\YMN\,$.

\begin{Proposition}
\label{TZ}
The series $Z(u)$ is invariant under the antiautomorphism \eqref{T} 
of\/ $\YMN\,$.
\end{Proposition}

\begin{proof}
The series $Z(u)$ is equal to the sum at the left hand side of
the relation \eqref{zetav} with $i=j\,$. By applying
the antiautomorphism \eqref{T} to that sum and using Proposition \ref{mst}
we get
\begin{equation*}
\sum_{k=1}^{M+N}
\TP_{ki}(u)\,T_{ik}(u+M-N)\,
{(-1)}^{\,\bk(\,\bi\ts+\ts1)\ts+\,\bi\ts(\,\bk\ts+\ts1)\ts+\ts
(\ts\bk\ts+\ts\bi\ts)(\ts\bi\ts+\ts\bk\ts)}
\end{equation*}
which is equal 
to the sum the left hand side of
the relation \eqref{zetau} with $i=j\,$.
\end{proof}

\begin{Corollary}
\label{STZ}
The automorphism \eqref{tusharp} of\/ $\YMN$ maps $Z(u)\mapsto Z(u)^{-1}\,$.
\end{Corollary}

\begin{proof}
The automorphism \eqref{tusharp} is the composition of the antiautomorphisms
\eqref{S} and \eqref{T} of $\YMN\,$. Hence the required statement follows 
from Corollary \ref{DZ} and Proposition \ref{TZ}.  
\end{proof}

Due to the definitions \eqref{tunat} and \eqref{zetav}
the coefficient of the series $Z(u)$  at $u^{-1}$ is zero. Thus
\begin{equation}
\label{zud}
Z(u)=1+Z^{\ts(2)}\ts u^{-2}+Z^{\ts(3)}\ts u^{-3}+\ldots
\end{equation}
for certain central elements $Z^{\ts(2)},Z^{\ts(3)},\ldots\in\YMN\,$.
The main result of the present section is

\begin{Theorem}
\label{Z}
The elements $Z^{\ts(2)},Z^{\ts(3)},\ldots$ are free generators
of the centre of\/ $\YMN\,$.
\end{Theorem}

We shall prove Theorem \ref{Z} at the end of this section.
We will use the following proposition.

\begin{Proposition}
\label{P28}
For any $r\geqslant2$ the element 
$Z^{(r)}\in\YMN$ has degree $r-2$
relative to the second filtration on $\YMN\,$. Its image
in the graded algebra $\grpr\YMN$ is equal to
\begin{equation}
\label{zim}
(\ts1-r\ts)\,\sum_{i=1}^{M+N}\,Y^{\ts(r-1)}_{ii}\,(-1)^{\,\bi}\,.
\end{equation}
\end{Proposition}

\begin{proof}
Let us expand the relation \eqref{trater} of series with
coefficients in $(\End\CMN)^{\ts\ot\ts2}\ot\YMN$ by using
the basis of $\End((\CMN)^{\ts\ot\ts2})$ constituted by the
elements $E_{ij}\ot E_{kl}\,$. 
By taking $(-1)^{\,\bi+\ts\bj}$ times the terms in the expansion 
corresponding to the basis element $E_{ij}\ot E_{ij}$ 
with any given indices $i$ and $j$ we obtain the relation of series
with coefficients in the algebra $\YMN$ 
\begin{align}
\notag
&T_{ij}(u)\,\TP_{ji}(v)+
\sum_{k=1}^{M+N}
T_{kj}(u)\,\TP_{jk}(v)\,(u-v-M+N)\1\,(-1)^{\,\bi}=
\\
\label{281}
&\TP_{ji}(v)\,T_{ij}(u)\,(-1)^{\,\bi+\ts\bj}+
\sum_{k=1}^{M+N}
\TP_{ki}(v)\,T_{ik}(u)\,(u-v-M+N)\1(-1)^{\,\bi}\,.
\end{align}
For any $i$ and $j$ denote by $\TD_{ij}(u)$ 
the formal derivative of the series $T_{ij}(u)$ so that
\begin{equation}
\label{TD}
\TD_{ij}(u)=
-\,T^{\ts(1)}_{ij}u^{-2}
-2\,T^{\ts(2)}_{ij}u^{-3}-\ldots\,.
\end{equation}
By tending in the relation \eqref{281} the parameter $u$ to $v+M-N$ 
we then get
\begin{align}
\notag
&T_{ij}(v+M-N)\,\TP_{ji}(v)+
\sum_{k=1}^{M+N}
\TD_{kj}(v+M-N)\,\TP_{jk}(v)\,(-1)^{\,\bi}=
\\
\label{282}
&\TP_{ji}(v)\,T_{ij}(v+M-N)\,(-1)^{\,\bi+\ts\bj}+
\sum_{k=1}^{M+N}
\TP_{ki}(v)\,\TD_{ik}(v+M-N)\,(-1)^{\,\bi}\,.
\end{align}

Let us now observe that
\begin{equation*}
T_{ij}(v+M-N)=T_{ij}(v)+(M-N)\,\TD_{ij}(v)+O_{ij}(v)
\end{equation*}
where $O_{\ts ij}(v)$
is a certain formal power series in $v\1$ with coefficients in $\YMN$
such that the coefficient at $v^{-r}$ with $r\ge3$
has degree $r-3$ relative to the second fitration.
The coefficient of this series at $v^{-r}$ with any $r\le2$ is zero.
By taking the sum of the relations \eqref{282} over the indices
$i=1\lc M+N$ and then using the definition \eqref{zetav} 
along with this observation we get
\begin{align*}
&Z(v)+(M-N)\sum_{k=1}^{M+N}\TD_{kj}(v+M-N)\,\TP_{jk}(v)=
\sum_{i,k=1}^{M+N}
\TP_{ki}(v)\,\TD_{ik}(v+M-N)\,(-1)^{\,\bi}\,+
\\
&\sum_{i=1}^{M+N}
\TP_{ji}(v)\,
(T_{ij}(v)+(M-N)\,\TD_{ij}(v)+O_{\ts ij}(v))\,(-1)^{\,\bi+\ts\bj}\,.
\end{align*}
By using the definition \eqref{tunat} the latter relation can be rewritten as
\begin{align}
\notag
&Z(v)+(M-N)\sum_{k=1}^{M+N}\TD_{kj}(v+M-N)\,\TP_{jk}(v)=
\sum_{i,k=1}^{M+N}
\TP_{ki}(v)\,\TD_{ik}(v+M-N)\,(-1)^{\,\bi}\,+
\\
\label{283}
&1+\sum_{i=1}^{M+N}
\TP_{ji}(v)\,((M-N)\,\TD_{ij}(v)+O_{\ts ij}(v))\,(-1)^{\,\bi+\ts\bj}\,.
\end{align}

For any indices $i$ and $j$ the leading term of the series $\TP_{ij}(v)$
is $\de_{ij}$ while for every $r\ge1$ the coefficient of this series at 
$v^{-r}$ has degree $r-1$ relative to the second filtration on $\YMN\,$. 
Furthermore for any given $r\ge1$
the coefficients at $v^{-r}$ of the series 
$\TD_{ij}(v)$ and $\TD_{ij}(v+M-N)$ have the same image
in the graded algebra $\grpr\YMN\,$, see \eqref{TD}.
Therefore by 
taking the coefficients at $v^{-r}$ with any $r\ge2$
in the relation \eqref{283}
the image of $Z^{(r)}$ in $\grpr\YMN$ equals
\begin{align*}
&-(M-N)\sum_{k=1}^{M+N}(1-r)\,Y^{(r-1)}_{kj}\,\de_{jk}\,+
\sum_{i,k=1}^{M+N}
\de_{ki}\,(1-r)\,Y^{(r-1)}_{ik}\,(-1)^{\,\bi}\,+
\\
&\sum_{i=1}^{M+N}
\de_{ji}\,(M-N)\,(1-r)\,Y^{(r-1)}_{ij}\,(-1)^{\,\bi+\ts\bj}
=
(\ts1-r\ts)\,\sum_{i=1}^{M+N}\,Y^{\ts(r-1)}_{ii}\,(-1)^{\,\bi}\,.
\tag*{\qed}
\end{align*}
\renewcommand{\qed}{}
\end{proof}

In our proof of Theorem \ref{Z} we will also use the following general 
proposition. 
Let $\a$ be any finite-dimensional Lie superalgebra over $\CC$.
Take the 
polynomial current Lie superalgebra $\a\ts[u]\,$.

\begin{Proposition}
\label{L27}
Suppose the centre of the Lie superalgebra $\a$ is trivial.
Then the centre of the universal enveloping algebra $\Ua$
is also trivial, that is equal to\/ $\CC\,$.
\end{Proposition}

\begin{proof}
Consider adjoint action of the Lie superalgebra $\a\ts[u]$ 
on its supersymmetric algebra. By
Poincar\'e\ts-Birkhoff-Witt theorem for $\Ua$
it suffices to show that the space
of invariants of~this action is trivial,
see \cite[Theorem 5.15]{MM}.

Let $A$ be an element of the supersymmetric algebra of $\a\ts[u]$
invariant under the adjoint action. Let $K=\dim\a\,$.
Choose any $\ZZ_2\ts$-homogeneous basis $f_1,\ldots,f_K$ of $\a$ and 
write the Lie brackets~as
\begin{equation*}
[\,f_p\,,f_q\,]=\sum_{r=1}^K C_{pq}^{\,r}\,f_r
\end{equation*}
where $C_{pq}^{\,r}\in\CC$ are the structure constants. 
Let $d$ be the maximal degree such that $A$ depends on at least one
of the basis 
elements $f_1\ts u^{\ts d}\lc f_K\ts u^{\ts d}$ of $\a\ts[u]\,$. Write
\begin{equation}
\label{A}
A\,=\sum_{m_1,\ldots,m_K}A_{\,m_1\ldots\ts m_K}\,
(\ts f_1\ts u^{\ts d}\ts)^{m_1}\ldots(\ts f_K\ts u^{\ts d}\ts)^{m_K}
\end{equation}
where the coefficients $A_{\,m_1\ldots\ts m_K}$ are certain elements
of the supersymmetric algebra which do~not depend
on $f_1\ts u^{\ts d}\lc f_K\ts u^{\ts d}\,$.
For each $p=1\lc K$ let $\bp$ be the $\ZZ_2\ts$-degree of the basis element 
$f_{\ts p}$ of $\a\,$. We allow $m_1\lc m_K$ in \eqref{A}
to range over $0,1,2,\,\ldots$ but assume that $A_{\,m_1\ldots\ts m_K}=0$
if $m_{\ts p}>1$ for at least one index $p$ with $\bp=1\,$. 

For each $p=1\lc K$ we have the equation 
$\text{ad}\ts(\ts f_p\ts u\ts)(A)=0$ in the supersymmetric algebra
of $\a\ts[u]\,$. In particular,
the component of the left hand side of this equation
that involves any of the basis elements 
$f_1\ts u^{\ts d+1}\lc f_K\ts u^{\ts d+1}$ 
must be zero. Thus
\begin{equation*}
\sum_{m_1,\ldots,m_K}A_{\,m_1\ldots\ts m_K}\,
\sum_{q,r=1}^K\,
m_{\ts q}\,C_{pq}^{\,r}\,
(\ts f_1\ts u^{\ts d}\ts)^{m_1}
\ldots
(\ts f_q\ts u^{\ts d}\ts)^{m_{\ts q}-1}
\ldots
(\ts f_K\ts u^{\ts d}\ts)^{m_K}\,
f_r\ts u^{\ts d+1}\,
(-1)^{\ts h_{\ts q}}=0
\end{equation*}
where we used the notation 
\begin{equation}
\label{hq}
h_{\ts q}=
\sum_{s=1}^{q-1}\,\,\bp\ts\bs\ts m_{\ts s}
\,+\!
\sum_{s=q+1}^K\br\ts\bs\ts m_{\ts s}\,.
\end{equation}
If follows that for any non-negative integers 
$m_1^{\,\prime},\ldots,m_K^{\,\prime}$ we have the equations 
\begin{equation}
\label{ad}
\sum_{q=1}^M\,
A_{\,m_1^{\,\prime}\ldots\ts m^{\,\prime}_q+1\ldots\ts m^{\,\prime}_K}
(m^{\,\prime}_q+1)\,C_{pq}^{\,r}\,(-1)^{\ts h_{\ts q}^{\ts\prime}}=0
\quad\text{where}\quad
p\,,r=1\lc K\,.
\end{equation}
Similarly to \eqref{hq} here we used the notation
\begin{equation*}
h_{\ts q}^{\ts\prime}=
\sum_{s=1}^{q-1}\,\,\bp\ts\bs\ts m_{\ts s}^{\,\prime}
\,+\!
\sum_{s=q+1}^K\br\ts\bs\ts m_{\ts s}^{\,\prime}\,.
\end{equation*}

Let us now fix any non-negative integers
$m_1^{\,\prime},\ldots,m_K^{\,\prime}$ 
and observe that the elements
\begin{equation*}
f_q^{\,\prime}=(m_q^{\,\prime}+1)\,f_q
\,(-1)^{\ts h_{\ts q}^{\ts\prime}}
\quad\text{with}\quad q=1\lc K
\end{equation*}
also form a basis of $\a\,$. Since the centre of 
the Lie superalgebra $\a$ is trivial, the system
\begin{equation*}
[\,f_p\,,\,\sum_{q=1}^K\, w_{\ts q}\,f_q^{\,\prime}\,]=0
\quad\text{with}\quad
p=1\lc K
\end{equation*}
of linear equations on 
$w_{\ts 1}\lc w_{\ts K}\in\CC$
has only zero solution. This system can be written~as
\begin{equation*}
\sum_{q=1}^K\,w_{\ts q}\,
(m^{\,\prime}_q+1)\,C_{pq}^{\,r}\,(-1)^{\ts h_{\ts q}^{\ts\prime}}=0
\quad\text{where}\quad
p\,,r=1\lc K\,.
\end{equation*}
Hence by comparing the latter system with \eqref{ad} we obtain that 
$A_{\,m_1^{\,\prime}\ldots\ts m^{\,\prime}_q+1\ldots\ts m^{\,\prime}_K}=0$
for each index $q=1\lc K\,$. It follows that $A\in\CC$
and that $d=0$ in particular.
\end{proof}

We can now prove Theorem \ref{Z}.
Due to Proposition \ref{zu} the elements 
\begin{equation*}
Z^{\ts(2)},Z^{\ts(3)},\ldots\in\YMN
\end{equation*}
are central. Therefore it suffices to
prove that the images of these elements in the graded algebra
$\grpr\YMN$ are free generators of its centre, see Proposition \ref{P28}.
By Theorem \ref{T18} the graded algebra 
is isomorphic to the universal enveloping algebra 
$\Ug$ via \eqref{UY}. Under this isomorphism the element \eqref{zim} 
of $\grpr\YMN$ with any $r\ge2$ corresponds to the element
\begin{equation*}
(\ts r-1\ts)\,\sum_{i=1}^{M+N}e_{ii}\,u^{\ts r-2}\in\Ug\,.
\end{equation*}
Let us demonstrate that all the latter elements
are free generators of the centre of $\Ug\,$.
They are algebraically independent
by the Poincar\'e\ts-Birkhoff-Witt theorem for 
$\Ug\,$, see \cite[Theorem 5.15]{MM}.
To show that they generate the centre of $\Ug$ consider 
the quotient of $\Ug$ by the ideal 
they generate. This quotient  
is isomorphic to the universal enveloping algebra $\Ua$
where the Lie superalgebra $\a$ is the quotient of 
$\glMN$ by the span of the element \eqref{esum}.
The centre of the Lie superalgebra $\a$ is trivial.
Hence the centre of 
$\Ua$ is also trivial, see Proposition \ref{L27}.
This argument completes our proof of Theorem \ref{Z}.
 

\section{Quantum Berezinian}
\label{sec:3}

The element \eqref{tu} can be also regarded as an $(M+N)\times(M+N)$ matrix
whose $ij$ entry is the series $T_{ij}(u)\,$.  
The \emph{quantum Berezinian} of that matrix is the series $B(u)$ defined
as the product
\begin{align*}
&\sum_{\si\in\Sym_M} \,(-1)^\si\,\ts
T_{\si(1)1}(u+M-N-1)\,T_{\si(2)2}(u+M-N-2)\ldots T_{\si(M)M}(u-N)\ \times
\\
&\sum_{\si\in\Sym_N} \,(-1)^\si\,\ts
\TP_{M+1,M+\si(1)}(u-N)\,
\TP_{M+2,M+\si(2)}(u-N+1)\ldots
\TP_{M+N,M+\si(N)}(u-1)
\end{align*}
of two alternated sums over the symmetric groups $\Sym_M$ and $\Sym_N\,$. 
Here $(-1)^\si$ denotes the sign of the permutation $\si\,$. 
The main purpose of the present section is to prove the following theorem.

\begin{Theorem}
\label{BZ}
We have the equality $B(u+1)=Z(u)\,B(u)$ of formal power series in $u\1\ts$.
\end{Theorem}

Due to the definition \eqref{3.0} the leading term
of the formal power series $T_{ij}(u)$ in $u\1$ is $\de_{ij}\,$.
Due to the definition 
\eqref{tunat} the leading term
of the series $\TP_{ij}(u)$ is also $\de_{ij}\,$.
It follows that the leading term
of 
$B(u)$ is $1\,$.
Using this observation, the coefficients of
the series $B(u)$ are uniquely determined by those of the series $Z(u)$
via the equality in Theorem \ref{BZ},
see the expansion \eqref{zud}.

Theorems \ref{Z} and \ref{BZ} now imply
that the coefficients of the series $B(u)$ at $u\1,u^{-2},\ts\ldots$
are free generators of the centre of $\YMN\,$. 
Corollary \ref{DZ} and Theorem \ref{BZ} now imply that
\begin{equation*}
\De:\,B(u)\mapsto B(u)\ot B(u)\,,
\ \quad
\ep:\,B(u)\mapsto1\,,
\ \quad
\S:\,B(u)\mapsto B(u)^{-1}\,.
\end{equation*}
Note that 
the second assignment here also follows directly from the definition of $B(u)$
because 
\begin{equation*}
\ep:T_{ij}(u)\mapsto\de_{ij}
\quad\text{and}\quad
\ep:\TP_{ij}(u)\mapsto\de_{ij}\,.
\end{equation*}

For $N=0$ the Hopf algebra $\YMN$ is
the \textit{Yangian} $\operatorname{Y}(\mathfrak{gl}_M)$
of the Lie algebra $\mathfrak{gl}_M\,$,~see~\cite{MNO}. 
In this case the second sum in the above 
definition of $B(u)$ is assumed to be $1\,$, and $B(u)$~equals
\begin{align}
\label{qdet}
&\sum_{\tau\in\Sym_M} \,(-1)^\si\,
T_{\si(1)\ts1}(u+M-1)\,T_{\si(2)\ts2}(u+M-2)\ldots T_{\si(M)\ts M}(u)\,.
\end{align}
This is the \emph{quantum determinant} of the $M\times M$ matrix 
whose $ij$ entry is the series $T_{ij}(u)\,$.
It~has been well known that 
the coefficients of the series \eqref{qdet} 
at $u\1,u^{-2},\ts\ldots$ 
are free generators of the centre of $\operatorname{Y}(\mathfrak{gl}_M)\,$,
see \cite[Section 2]{MNO} and references therein.
A detailed proof of Theorem~\ref{BZ} in this particular case was given in
\cite[Section 5]{MNO} by following~\cite{N1}.

For any $M\,,N$ denote by $X(u)$ the $(M+N)\times(M+N)$ matrix
whose $ij$ entry is the series 
\begin{equation*}
X_{ij}(u)=
\de_{ij}\cdot1+X_{ij}^{\ts(1)}\ts u^{-1}+X_{ij}^{\ts(2)}\ts u^{-2}+\ldots
\end{equation*}
times $(-1)^{\ts\,(\,\bi+1\ts)\,\bj}$
with coefficients of $\ZZ_2\ts$-degree $\bi+\bj\,$. See
Section \ref{sec:1} for the definition of these coefficients.
The matrix $X(u)$ is invertible.
Let $\XP_{ij}(u)$ be the $ij$ entry of the inverse matrix.
Under the correspondence $\YMN\to\gr\YMN$
the series $B(u)$ gets mapped to the product
\begin{align*}
&\sum_{\si\in\Sym_M} \,(-1)^\si\,
X_{\si(1)1}(u)\,X_{\si(2)2}(u)\ldots X_{\si(M)M}(u)\ \times
\\
&\sum_{\si\in\Sym_N} \,(-1)^\si\,
\XP_{M+1,M+\si(1)}(u)\,\XP_{M+2,M+\si(2)}(u)\ldots\XP_{M+N,M+\si(N)}(u)
\end{align*}
of two determinants. This is just the \textit{Berezinian} or the
\emph{superdeterminant}
of the matrix $X(u)$ as defined in \cite[Section I.3.1]{B}.
These two observations explain our choice of 
terminology for 
$B(u)\,$.  

Let us now consider the other particular case when $M=0\,$.
In this case $B(u)$ equals the sum
\begin{align*}
&\sum_{\si\in\Sym_N} \,(-1)^\si\,
\TP_{1\ts\si(1)}(u-N)\,\TP_{2\ts\si(2)}(u-N+1)\ldots\TP_{N\ts\si(N)}(u-1)\,.
\end{align*}
This sum can also be obtained by applying the automorphism 
\eqref{tusharp} of $\operatorname{Y}(\mathfrak{gl}_{\,0\ts|N})$ to the series
\begin{align*}
\label{qdet}
&\sum_{\si\in\Sym_N} \,(-1)^\si\,
T_{\si(1)\ts1}(u-N)\,T_{\si(2)\ts2}(u-N-1)\ldots T_{\si(N)\ts N}(u-1)\,.
\end{align*}
Let us denote by $C(u)$ the latter series. Then due to Corollary \ref{STZ}
the statement of Theorem~\ref{BZ} for $M=0$ becomes equivalent to the relation 
\begin{equation}
\label{AZ}
Z(u)\,C(u+1)=C(u)\,.
\end{equation}

Now observe that the Yangian $\operatorname{Y}(\mathfrak{gl}_{\,0\ts|N})$
is isomorphic to 
$\operatorname{Y}(\mathfrak{gl}_{N|\ts0}\ts)=
\operatorname{Y}(\mathfrak{gl}_N)\,.$
A Hopf algebra isomorphism 
$\operatorname{Y}(\mathfrak{gl}_{\,0\ts|N})\to
\operatorname{Y}(\mathfrak{gl}_{N}\ts)$
can be defined by the assignment $T_{ij}(u)\mapsto T_{ij}(-u)\,$,
see~\eqref{3.1}. Under this isomorphism $Z(u)\mapsto Z(-u)\,$, see 
\eqref{zetav}. Denote by $D(u)$ the quantum determinant for
the Yangian
$\operatorname{Y}(\mathfrak{gl}_{N})\,$. This is the series
\eqref{qdet} with $M$ replaced by $N\ts$. Then $C(u)\mapsto D\ts(1-u)$
under the isomorphism. Therefore by applying the isomorphism to the 
relation \eqref{AZ} we get
\begin{equation*}
Z(-u)\,D(-u)=D\ts(1-u)\,.
\end{equation*}
The latter relation holds by 
Theorem \ref{BZ} for 
$\operatorname{Y}(\mathfrak{gl}_{N}\ts)\,$.
Hence Theorem 3.1 also holds for $M=0\,$.

From now on until the end of this section 
we will be assuming that $M\ts,N>0\,$.
The next proposition goes back  to \cite[Theorem 2.4]{C}, see also 
\cite[Section 1]{G2}.
In particular, this proposition implies 
that the two alternated sums in the definition of 
$B(u)$ commute with each other. 

\begin{Proposition}
\label{L3}
If\/ $i,j\le M<k,l$ then the coefficients of the series
$T_{ij}(u)$ commute with the coefficients of the series $\TP_{kl}(v)\,$.
\end{Proposition}

\begin{proof}
Let $I\in\End\CMN$ and $J\in\End\CMN$ 
be the projections of the $\ZZ_2$-graded vector space $\CMN$ onto 
its even and odd subspaces respectively, so that 
$I\ts e_i=\de_{\ts\ts0\ts\ts\bi}\,e_i$ and 
$J\ts e_i=\de_{\ts\ts1\ts\ts\bi}\,e_i$
for every index $i=1\lc M+N\ts$. 
These two subspaces of $\CMN$ are denoted by 
$\CC^{M|\ts0}$ and $\CC^{\,0\ts|N}$.
By the definition \eqref{Q} we have the relations in 
$(\End\CMN)^{\ot2}$
\begin{equation}
\label{IJQ}
(I\ot J)\,Q=
Q\,(I\ot J)=0\,.
\end{equation}
Hence by multiplying \eqref{trater} 
on the left and on the right by $I\ot J\ot1$ we get 
the relation 
\begin{equation}
\label{referee2}
(I\ot J\ot1)\,T_1(u)\,\TS_2(v)\,(I\ot J\ot1)=
(I\ot J\ot1)\,\TS_2(v)\,T_1(u)\,(I\ot J\ot1)
\end{equation}
of series in $u\,,v$ with coefficients 
in the algebra $(\End\CMN)^{\ot2}\ot\YMN\,$,
see \eqref{rutinv}. 
The latter relation is equivalent to the collection of all
commutation relations stated in Proposition~\ref{L3}.
\end{proof}

We will keep using the projectors $I$ and $J$ introduced in
the proof of Proposition \ref{L3} above.
Note that they also satisfy the relations in the algebra $(\End\CMN)^{\ot2}$
\begin{align*}
(J\ot I)\,Q&=Q\,(J\ot I)=0\,,
\\[4pt]
Q\,(I\ot I)&=Q\,(I\ot 1)=Q\,(1\ot I)
\quad\text{and}\quad
Q\,(J\ot J)=Q\,(J\ot 1)=Q\,(1\ot J)\,.
\end{align*}
These relations together with \eqref{IJQ} imply that  
\begin{equation}
Q=Q\,(I\ot I+I\ot J+J\ot I+J\ot J)=Q\,(I\ot I+J\ot J)\,=Q\,(I\ot 1+1\ot J)\,. 
\label{IJ}
\end{equation}  
The next two technical propositions
will be employed in our proof of Theorem \ref{BZ} for $M\ts,N>0\,$. 

\begin{Proposition}
\label{L2}
We have the equality in the algebra $(\End\CMN)^{\ot\ts(M+N+2)}$ 
\begin{align*}
&Q_{\,1,M+N+2}\,
\Bigl(\ts1-\frac1M\,Q_{\,M+1,M+N+2}\ts\Bigr)\,
\Bigl(\ts1+\frac1N\,Q_{\,1,M+2}\ts\Bigr)\,\times
\\[8pt]
&I_{\ts1}\ts\ldots\ts I_{\ts M}\,(\,I_{\ts M+1}+J_{\ts M+2}\,)\,
J_{\ts M+3}\ts\ldots\ts J_{\ts M+N+2}\,=
\\[6pt]
&I_{\ts2}\ts\ldots\ts I_{\ts M+1}\,
J_{\ts M+2}\ts\ldots\ts J_{\ts M+N+1}\,Q_{\,1,M+N+2}\,
\Bigl(-\,\frac1M\,P_{\,1,M+1}\,J_{\ts M+N+2}+
\frac1N\,P_{\,M+2,M+N+2}\,I_{\ts1}\ts\Bigr)\ts.
\end{align*}
\end{Proposition}

\begin{proof}
The second relation in \eqref{IJQ} implies that
\begin{align*}
&Q_{\,1,M+N+2}\,I_{\ts1}\,J_{\ts M+N+2}=0\,,
\\[8pt]
&Q_{\,M+1,M+N+2}\,I_{\ts M+1}J_{\ts M+N+2}=0\,,
\\[8pt]
&Q_{\,1,M+2}\,I_{\ts1}\,J_{\ts M+2}=0\,.
\end{align*}
Therefore by opening the parentheses at the left hand side of 
required equality we get the sum
\begin{align}
\notag
-\,\frac1M\,&
Q_{\,1,M+N+2}\,Q_{\,M+1,M+N+2}\,
I_{\ts1}\ts\ldots\ts I_{\ts M}\,
J_{\ts M+2}\,
J_{\ts M+3}\ts\ldots\ts J_{\ts M+N+2}
\\[8pt]
\label{referee1}
+\,\frac1N\,&
Q_{\,1,M+N+2}\,Q_{\,1,M+2}\,
I_{\ts1}\ts\ldots\ts I_{\ts M}\,
I_{\ts M+1}\,
J_{\ts M+3}\ts\ldots\ts J_{\ts M+N+2}\,\ts.
\end{align}

We have the relation
$P_{\ts 23}\,P_{\ts 13}=P_{\ts 13}\,P_{\ts 12}$ in 
$(\End\CMN)^{\ot3}\,$.
By applying to this relation the antiautomorphism $\tau$ 
of the third tensor factor $\End\CMN$ we get the relation 
$Q_{\ts 13}\,Q_{\ts 23}=Q_{\ts 13}\,P_{\ts 12}\,$.
It follows that in the algebra $(\End\CMN)^{\ot\ts(M+N+2)}$ 
\begin{equation}
\label{QQP}
Q_{\,1,M+N+2}\,Q_{\,M+1,M+N+2}=Q_{\,1,M+N+2}\,P_{\,1,M+1}\,.
\end{equation}
Further, since $(\tau\ot\tau)(P)=P$ we have the equality
\begin{equation*}
Q=(\ts\tau\1\ot\id\ts)\ts(\ts P\ts)
\end{equation*}
by the definition \eqref{Q}.
Hence applying to the relation $P_{\ts 12}\,P_{\ts 13}=P_{\ts 13}\,P_{\ts 23}$
the antiautomorphism $\tau\1$ of the first tensor factor of 
$(\End\CMN)^{\ot3}$ yields the relation
$Q_{\ts 13}\,Q_{\ts 12}=Q_{\ts 13}\,P_{\ts 23}\,$.
It follows that in the algebra $(\End\CMN)^{\ot\ts(M+N+2)}$ 
\begin{equation}
\label{QQQ}
Q_{\,1,M+N+2}\,Q_{\,1,M+2}=Q_{\,1,M+N+2}\,P_{\,M+2,M+N+2}\,.
\end{equation}
Therefore the sum displayed in the two lines of \eqref{referee1} equals
\begin{align*}
-\,\frac1M\,&
Q_{\,1,M+N+2}\,P_{\,1,M+1}\,
I_{\ts1}\ts\ldots\ts I_{\ts M}\,
J_{\ts M+2}\,
J_{\ts M+3}\ts\ldots\ts J_{\ts M+N+2}\,
\\[8pt]
+\,\frac1N\,&
Q_{\,1,M+N+2}\,P_{\,M+2,M+N+2}\,
I_{\ts1}\ts\ldots\ts I_{\ts M}\,
I_{\ts M+1}\,
J_{\ts M+3}\ts\ldots\ts J_{\ts M+N+2}\,\ts.
\end{align*}
This can be written as the right hand side of the equality in
Proposition \ref{L2} by using the relations
\begin{align*}
&P_{\,1,M+1}\,I_{\ts 1}=I_{\ts M+1}\,P_{\,1,M+1}\,,
\\[8pt]
&P_{\,M+2,M+N+2}\,J_{\ts M+N+2}=J_{\ts M+2}\,P_{\,M+2,M+N+2}\,.
\tag*{\qed}
\end{align*}
\renewcommand{\qed}{}
\end{proof}

For any positive integer $n$ denote respectively by 
$G^{\ts(n)}$ and $H^{(n)}$ the images of the elements 
\begin{equation*}
\sum_{\si\in\Sym_n}(-1)^\si\,\si
\quad\text{and}\quad
\sum_{\si\in\Sym_n}\si
\end{equation*}
of the group ring $\CC\ts\Sym_n$
in the algebra 
$(\End\CMN)^{\ot\ts n}\,$. 
We use the representation $\si_m\mapsto P_{m,m+1}$ 
for $m=1\lc n-1$ introduced in Section \ref{sec:1}. 
Note the relations in the algebra 
$(\End\CMN)^{\ot\ts n}$ 
\begin{align}
\label{GG}
G^{\ts(n)}&\!=
(1-P_{\,1,n}-\ldots-P_{\,n-1,n})\,(G^{\ts(n-1)}\ot1)\,,
\\[8pt]
\label{HH}
H^{(n)}&\!=
(1+P_{\,1,n}+\ldots+P_{\,n-1,n})\,(H^{(n-1)}\ot1)\,.
\end{align}
Here we assume that $G^{\ts(0)}=H^{(0)}=1\,$.
To avoid cumbrous notation, below we will simply write
$G$ for $G^{\ts(M)}$ and $H$ for $H^{(N)}\ts$.
These two images act by antisymmetrization on the subspaces 
\begin{equation}
\label{SS}
(\CC^{M|\ts0})^{\ot\ts M}\subset(\CMN)^{\ot\ts M}
\quad\text{and}\quad
(\CC^{\,0\ts|N})^{\ot\ts N}\subset(\CMN)^{\ot\ts N}\,.
\end{equation}

\begin{Proposition}
\label{L1}
We have the equality in the algebra 
$(\End\CMN)^{\ot\ts(M+N+2)}$ 
\begin{align*}
&I_{\ts2}\ts\ldots\ts I_{\ts M+1}\,
J_{\ts M+2}\ts\ldots\ts J_{\ts M+N+1}\,
G_{\,2\ts\ldots\ts M+1}\,H_{\ts M+2\ts\ldots\ts M+N+1}\,
Q_{\,1,M+N+2}\ \times
\\[8pt]
&\Bigl(\ts1-\frac1M\,Q_{\,M+1,M+N+2}\ts\Bigr)\,
\Bigl(\ts1+\frac1N\,Q_{\,1,M+2}\ts\Bigr)\,
G_{\,1\ts\ldots\ts M}\,H_{\ts M+3\ts\ldots\ts M+N+2}\,=
\\[8pt]
&(M-1)\ts!\,(N-1)\ts!\,
P_{\,1,M+1}\,P_{\,M+2,M+N+2}\ \times
\\[12pt]
&I_{\ts1}\ts\ldots\ts I_{\ts M}\,
J_{\ts M+3}\ts\ldots\ts J_{\ts M+N+2}\,
G_{\,1\ts\ldots\ts M}\,H_{\ts M+3\ts\ldots\ts M+N+2}\,\ts
Q_{\,M+1,M+2}\,\ts.
\end{align*}
\end{Proposition}

\begin{proof}
By again using the relation \eqref{QQP} we get the equalities
\begin{align*}
&I_{\ts M+1}\,J_{\ts M+2}\,Q_{\,1,M+N+2}\,Q_{\,M+1,M+N+2}\,Q_{\,1,M+2}=
I_{\ts M+1}\,J_{\ts M+2}\,Q_{\,1,M+N+2}\,P_{\,1,M+1}\,Q_{\,1,M+2}=
\\[12pt]
&I_{\ts M+1}\,J_{\ts M+2}\,Q_{\,1,M+N+2}\,Q_{\,M+1,M+2}\,P_{\,1,M+1}=
I_{\ts M+1}\,J_{\ts M+2}\,Q_{\,M+1,M+2}\,Q_{\,1,M+N+2}\,P_{\,1,M+1}\,.
\end{align*}
The product at the right hand side of these equalities 
vanishes by \eqref{IJQ}.
Hence by opening the parentheses at the left hand side of the equality
in Proposition \ref{L1} and using \eqref{QQP},\eqref{QQQ} we get
\begin{align}
&
I_{\ts2}\ts\ldots\ts I_{\ts M+1}\,
J_{\ts M+2}\ts\ldots\ts J_{\ts M+N+1}\,
G_{\,2\ts\ldots\ts M+1}\,H_{\ts M+2\ts\ldots\ts M+N+1}\,
Q_{\,1,M+N+2}\ \times
\notag
\\[8pt]
&
\Bigl(\ts1-\frac1M\,Q_{\,M+1,M+N+2}+\frac1N\,Q_{\,1,M+2}\ts\Bigr)\,
G_{\,1\ts\ldots\ts M}\,H_{\ts M+3\ts\ldots\ts M+N+2}\,=
\notag
\\[8pt]
&
I_{\ts2}\ts\ldots\ts I_{\ts M+1}\,
J_{\ts M+2}\ts\ldots\ts J_{\ts M+N+1}\,
G_{\,2\ts\ldots\ts M+1}\,H_{\ts M+2\ts\ldots\ts M+N+1}\,
Q_{\,1,M+N+2}\ \times
\notag
\\[8pt]
&
\Bigl(\ts1-\frac1M\,P_{\,1,M+1}+\frac1N\,P_{\,M+2,M+N+2}\ts\Bigr)\,
G_{\,1\ts\ldots\ts M}\,H_{\ts M+3\ts\ldots\ts M+N+2}\,=
\notag
\\[8pt]
&
I_{\ts2}\ts\ldots\ts I_{\ts M+1}\,
J_{\ts M+2}\ts\ldots\ts J_{\ts M+N+1}\,
Q_{\,1,M+N+2}\,
(\,G_{\,2\ts\ldots\ts M+1}\,H_{\ts M+2\ts\ldots\ts M+N+1}
\notag
\\[8pt]
&
-\frac1M\,P_{\,1,M+1}\,
G_{\,1\ts\ldots\ts M}\,H_{\ts M+2\ts\ldots\ts M+N+1}
+\frac1N\,P_{\,M+2,M+N+2}\,
G_{\,2\ts\ldots\ts M+1}\,H_{\ts M+3\ts\ldots\ts M+N+2}\,)\ \times
\notag
\\[8pt]
&G_{\,1\ts\ldots\ts M}\,H_{\ts M+3\ts\ldots\ts M+N+2}\,\ts.
\label{chain}
\end{align}

Observe that 
$G_{\,1\ts\ldots\ts M}^{\,2}=M\ts!\,G_{\,1\ts\ldots\ts M}$
while due to \eqref{GG} the product 
$G_{\,2\ts\ldots\ts M+1}\,G_{\,1\ts\ldots\ts M}$ equals
\begin{align*}
&(\,1-P_{\,2,M+1}-\ldots-P_{\,M,M+1}\,)\,
G^{\,(M-1)}_{\,2\ts\ldots\ts M}\,G_{\,1\ts\ldots\ts M}=
\\[12pt]
&(M-1)\ts!\,(\,1-P_{\,2,M+1}-\ldots-P_{\,M,M+1}\,)\,
G_{\,1\ts\ldots\ts M}\,.
\end{align*}
Similarly
$H_{M+3\ts\ldots\ts M+N+2}^{\,2}=N\ts!\,H_{M+3\ts\ldots\ts M+N+2}$ while 
$H_{\ts M+2\ts\ldots\ts M+N+1}\,H_{\ts M+3\ts\ldots\ts M+N+3}$ equals
\begin{align*}
&(\,1+P_{\,M+2,M+3}+\ldots+P_{\,M+2,M+N+1}\,)\,
H^{\ts(N-1)}_{\ts M+3\ts\ldots\ts M+N+1}\,H_{\ts M+3\ts\ldots\ts M+N+2}=
\\[12pt]
&(N-1)\ts!\,(\,1+P_{\,M+2,M+3}+\ldots+P_{\,M+2,M+N+1}\,)\,
H_{\ts M+3\ts\ldots\ts M+N+2}\,.
\end{align*}
Here we employed \eqref{HH}.
Hence the right hand side of the equalities \eqref{chain}
can be rewritten as 
\begin{equation}
\label{MN}
(M-1)\ts!\,(N-1)\ts!\,
I_{\ts2}\ts\ldots\ts I_{\ts M+1}\,
J_{\ts M+2}\ts\ldots\ts J_{\ts M+N+1}\,
Q_{\,1,M+N+2}
\end{equation}
multiplied on the right by
\begin{align}
&
(\ts(\,1-P_{\,2,M+1}-\ldots-P_{\,M,M+1}\,)\,
(\,1+P_{\,M+2,M+3}+\ldots+P_{\,M+2,M+N+1}\,)
\notag
\\[12pt]
&
-\,P_{\,1,M+1}\,
(\,1+P_{\,M+2,M+3}+\ldots+P_{\,M+2,M+N+1}\,)
\notag
\\[12pt]
&
+P_{\,M+2,M+N+2}\,(\,1-P_{\,2,M+1}-\ldots-P_{\,M,M+1}\,)\ts)\,
G_{\,1\ts\ldots\ts M}\,H_{\ts M+3\ts\ldots\ts M+N+2}\,=
\notag
\\[12pt]
&
(\ts(\,1-P_{\,1,M+1}-\ldots-P_{\,M,M+1}\,)\,
(\,1+P_{\,M+2,M+3}+\ldots+P_{\,M+2,M+N+2}\,)
\notag
\\[12pt]
&
+\,
P_{\,1,M+1}\,P_{\,M+2,M+N+2}\,)\,
G_{\,1\ts\ldots\ts M}\,H_{\ts M+3\ts\ldots\ts M+N+2}\,=
\notag
\\[12pt]
&
\label{GH}
G^{\,(M+1)}_{\,1\ts\ldots\ts M+1}\,H^{\ts(N+1)}_{\ts M+2\ts\ldots\ts M+N+2}
+
P_{\,1,M+1}\,P_{\,M+2,M+N+2}\,
G_{\,1\ts\ldots\ts M}\,H_{\ts M+3\ts\ldots\ts M+N+2}\,.
\end{align}

Due to \eqref{IJ} we have the equality
\begin{equation*}
Q_{\,1,M+N+2}=Q_{\,1,M+N+2}\,(\ts I_{\ts1}+J_{\ts M+N+2}\ts)\,.
\end{equation*}
Therefore by multiplying 
\eqref{MN} by the first summand at the
right hand side of 
\eqref{GH} we get
\begin{align*}
&
(M-1)\ts!\,(N-1)\ts!\,Q_{\,1,M+N+2}\,
(\,I_{\ts1}\ts\ldots\ts I_{\ts M+1}\,
J_{\ts M+2}\ts\ldots\ts J_{\ts M+N+1}
\\[8pt]
&
+\,I_{\ts2}\ts\ldots\ts I_{\ts M+1}\,
J_{\ts M+2}\ts\ldots\ts J_{\ts M+N+2}\,)\,
G^{\,(M+1)}_{\,1\ts\ldots\ts M+1}\,H^{\ts(N+1)}_{\ts M+2\ts\ldots\ts M+N+2}\,.
\end{align*}
The latter product vanishes because zero is the only
antisymmetric tensor in the subspaces
\begin{equation*}
(\CC^{M|\ts0})^{\ot\ts(M+1)}\subset(\CMN)^{\ot\ts(M+1)}
\quad\text{and}\quad
(\CC^{\,0\ts|N})^{\ot\ts(N+1)}\subset(\CMN)^{\ot\ts(N+1)}\,.
\end{equation*}
Multiplying 
\eqref{MN} by the second summand at the
right hand side of the equalities \eqref{GH} we get
\begin{align*}
&(M-1)\ts!\,(N-1)\ts!\,
I_{\ts2}\ts\ldots\ts I_{\ts M+1}\,
J_{\ts M+2}\ts\ldots\ts J_{\ts M+N+1}\,
Q_{\,1,M+N+2}\ \times
\\[12pt]
&P_{\,1,M+1}\,P_{\,M+2,M+N+2}\,
G_{\,1\ts\ldots\ts M}\,H_{\ts M+3\ts\ldots\ts M+N+2}
\end{align*}
which is equal to the product at the right hand side of the equality 
stated in Proposition \ref{L1}.
\end{proof}

\begin{Proposition}
\label{P}
For any positive integer $n$ 
we have equalities 
of series in $u$ with coefficients in the algebra
$(\End\CMN)^{\ot\ts n}\ot\YMN$
\begin{align*}
&(G^{\ts(n)}\ot1)\,T_1(u)\ts\ldots\ts T_n(u-n+1)\,=
T_n(u-n+1)\ts\ldots\ts T_1(u)\,(G^{\ts(n)}\ot1)\,,
\\[8pt] 
&(H^{(n)}\ot1)\,\TS_1(u)\ts\ldots\ts\TS_n(u+n-1)\,= 
\TS_n(u+n-1)\ts\ldots\ts\TS_1(u)\,(H^{\ts(n)}\ot1)\,.
\end{align*}
\end{Proposition}

\begin{proof}
If $1\le i<j\le n$ then let $\si_{ij}\in\Sym_n$ be the transposition of
the numbers $i$ and $j\,$. 
There is a well known identity in the symmetric group ring $\CC\ts\Sym_n$ 
\begin{equation}
\label{R1}
\sum_{\si\in\Sym_n}(-1)^\si\,\si\,=\,
\prod_{j=2}^n\,\Bigl(\ \prod_{i=1}^{j-1}\,\,\Bigl(\,1-\frac{\si_{ij}}{j-i}\,
\Bigr)\Bigr)
\end{equation}
where the factors at the right hand side are arranged from left to right
as $i$ and $j$ are increasing, see
for instance \cite[Section~2.3]{MNO}.
The identity implies the relation in the algebra $(\End\CMN)^{\ot\ts n}$
\begin{equation}
\label{RG}
G^{\ts(n)}\,=\,
\prod_{j=2}^n\,\Bigl(\ \prod_{i=1}^{j-1}\,\,R_{\ts ij}(j-i)\,\Bigr)\Bigr)\,.
\end{equation}
The first equality stated in Proposition \ref{P} follows from this relation
by repeatedly using \eqref{3.3}.

Further, \eqref{R1} is equivalent to another well known
identity in $\CC\ts\Sym_n$
\begin{equation*}
\sum_{\si\in\Sym_n}\si\,=\,
\prod_{j=2}^n\,\Bigl(\ \prod_{i=1}^{j-1}\,\,\Bigl(\,1+\frac{\si_{ij}}{j-i}\,
\Bigr)\Bigr)
\end{equation*}
which implies the relation in the algebra $(\End\CMN)^{\ot\ts n}$
\begin{equation}
\label{RH}
H^{(n)}\,=\,
\prod_{j=2}^n\,\Bigl(\ \prod_{i=1}^{j-1}\,\,R_{\ts ij}(i-j)\,\Bigr)\Bigr)\,.
\end{equation}
The second equality in Proposition \ref{P} follows from this relation
by repeatedly using \eqref{3.333}.
\end{proof}

We will now prove Theorem \ref{BZ}. 
Using the relations \eqref{3.3},\eqref{3.333} and \eqref{trater} 
we get an equality of series in $u$ with coefficients 
in the algebra $(\End\CMN)^{\ot\ts(M+N+2)}\ot\YMN$
\begin{align}
&(\ts\RS_{\,1,M+2}\ts(M\ts)\,
R_{\,1M}\ts(M-1)\ts\ldots\ts R_{\,12}\ts(1)\ot1\ts)\ \times
\notag
\\[8pt]
&(\ts\RS_{\,M+1,M+N+2}\ts(-N\ts)\,
R_{\,M+3,M+N+2}\ts(1-N)\ts\ldots\ts R_{\,M+N+1,M+N+2}\ts(-1)\ot1\ts)\ \times
\notag
\\[8pt]
&T_1\ts(u+M-N)\,T_2\ts(u+M-N-1)\ts\ldots\ts T_{M}\ts(u-N+1)\,
\TS_{M+2}\ts(u-N)\ \times
\notag
\\[8pt]
&T_{M+1}\ts(u-N)\,
\TS_{M+3}\ts(u-N+1)\ts\ldots\ts\TS_{M+N+1}\ts(u-1)\,\TS_{M+N+2}\ts(u)\,=
\notag
\\[8pt]
&T_2\ts(u+M-N-1)\ts\ldots\ts T_{M}\ts(u-N+1)\,
\TS_{M+2}\ts(u-N)\,T_1\ts(u+M-N)\ \times
\notag
\\[8pt]
&\TS_{M+N+2}\ts(u)\,T_{M+1}\ts(u-N)\,
\TS_{M+3}\ts(u-N+1)\ts\ldots\ts\TS_{M+N+1}\ts(u-1)\ \times
\notag
\\[8pt]
&(\ts\RS_{\,1,M+2}\ts(M\ts)\,
R_{\,1M}\ts(M-1)\ts\ldots\ts R_{\,12}\ts(1)\ot1\ts)\ \times
\notag
\\[8pt]
&(\ts\RS_{\,M+1,M+N+2}\ts(-N\ts)\,
R_{\,M+3,M+N+2}\ts(1-N)\ts\ldots\ts R_{\,M+N+1,M+N+2}\ts(-1)\ot1\ts)\,.
\label{E1}
\end{align}

Let us multiply both sides of the equality \eqref{E1}
respectively on the left and on the right by
\begin{equation*}
I_{\ts2}\ts\ldots\ts I_{\ts M+1}\,
J_{\ts M+2}\ts\ldots\ts J_{\ts M+N+1}\,
G_{\,2\ts\ldots\ts M}^{\,(M-1)}\,H_{\ts M+3\ts\ldots\ts M+N+1}^{\,(N-1)}\,
Q_{\,1,M+N+2}\ot1
\end{equation*}
and by $P_{\,1,M+1}\,P_{\,M+2,M+N+2}\ot1\,$. 
Due to \eqref{RG} and \eqref{RH} we have in 
$(\End\CMN)^{\ot\ts(M+N+2)}$
\begin{align}
\label{GM}
&G_{\,2\ts\ldots\ts M}^{\,(M-1)}\,R_{\,1M}\ts(M-1)\ts\ldots\ts R_{\,12}\ts(1)=
G_{\,1\ts\ldots\ts M}\,,
\\[8pt]
\label{HN}
&H_{\ts M+3\ts\ldots\ts M+N+1}^{\,(N-1)}\,
R_{\,M+3,M+N+2}\ts(1-N)\ts\ldots\ts R_{\,M+N+1,M+N+2}\ts(-1)=
H_{\ts M+3\ts\ldots\ts M+N+2}\ts\,.
\end{align}
Hence after the multiplication the left hand side of \eqref{E1} becomes
\begin{align}
&
(\ts I_{\ts2}\ts\ldots\ts I_{\ts M+1}\,
J_{\ts M+2}\ts\ldots\ts J_{\ts M+N+1}\,Q_{\,1,M+N+2}\,
\RS_{\,1,M+2}\ts(M\ts)\,
\RS_{\,M+1,M+N+2}\ts(-N\ts)\ot1\ts)\ \times
\notag
\\[8pt]
&
(\ts G_{\,1\ts\ldots\ts M}\ot1\ts)\,
T_1\ts(u+M-N)\,T_2\ts(u+M-N-1)\ts\ldots\ts T_{M}\ts(u-N+1)\,
\TS_{M+2}\ts(u-N)\ \times
\notag
\\[8pt]
&
(\ts H_{\ts M+3\ts\ldots\ts M+N+2}\ot1\ts)\,
T_{M+1}\ts(u-N)\,
\TS_{M+3}\ts(u-N+1)\ts\ldots\ts\TS_{M+N+1}\ts(u-1)\,\TS_{M+N+2}\ts(u)\ \times
\notag
\\[8pt]
&(\ts P_{\,1,M+1}\,P_{\,M+2,M+N+2}\ot1\ts)\,. 
\label{E2}
\end{align}
After the same multiplication the right hand side of \eqref{E1} becomes
\begin{align}
&(\ts I_{\ts2}\ts\ldots\ts I_{\ts M+1}\,
J_{\ts M+2}\ts\ldots\ts J_{\ts M+N+1}\,
G_{\,2\ts\ldots\ts M}^{\,(M-1)}\,H_{\ts M+3\ts\ldots\ts M+N+1}^{\,(N-1)}\,
Q_{\,1,M+N+2}\ot1\ts)\ \times
\notag
\\[8pt]
&T_2\ts(u+M-N-1)\ts\ldots\ts T_{M}\ts(u-N+1)\,
\TS_{M+2}\ts(u-N)\,T_1\ts(u+M-N)\ \times
\notag
\\[8pt]
&\TS_{M+N+2}\ts(u)\,T_{M+1}\ts(u-N)\,
\TS_{M+3}\ts(u-N+1)\ts\ldots\ts\TS_{M+N+1}\ts(u-1)\ \times
\notag
\\[8pt]
&(\ts\RS_{\,1,M+2}\ts(M\ts)\,
R_{\,1M}\ts(M-1)\ts\ldots\ts R_{\,12}\ts(1)\ot1\ts)\ \times
\notag
\\[8pt]
&(\ts\RS_{\,M+1,M+N+2}\ts(-N\ts)\,
R_{\,M+3,M+N+2}\ts(1-N)\ts\ldots\ts R_{\,M+N+1,M+N+2}\ts(-1)\ot1\ts)\ \times
\notag
\\[8pt]
&(\ts P_{\,1,M+1}\,P_{\,M+2,M+N+2}\ot1\ts)\,. 
\label{E4}
\end{align}

Now recall that the \emph{supertrace} on $\End\CMN$ is the linear function
defined by the assignment
\begin{equation*}
\str:E_{ij}\mapsto\de_{ij}\,(-1)^{\,\bi}\,.
\end{equation*} 
For any homogeneous elements $X,X'\in\End\CMN$ we have the equality
\begin{equation*}
\str\hspace{0.5pt}(X\ts X'\ts)=
\str\hspace{0.5pt}(X'X\ts)\,(-1)^{\ts\deg X\deg X^\prime}.
\end{equation*} 
Let us define a linear map
\begin{equation}
\label{STR}
(\End\CMN)^{\ot\ts(M+N+2)}\ot\YMN\to(\End\CMN)^{\ot\ts2}\ot\YMN
\end{equation}
as applying $\str$ to all tensor factors of 
$(\End\CMN)^{\ot\ts(M+N+2)}$ except the first and the last~ones.
We will relate elements of the source vector space in \eqref{STR}  
by the symbol $\sim$ if
their images by this map are the same.
We extend the relation
$\sim$ to series in $u$ with coefficients in the source.

By Proposition \ref{P} the product displayed in the four lines of \eqref{E2}
is divisible on the right by 
$
G_{\,2\ts\ldots\ts M+1}\,H_{\ts M+2\ts\ldots\ts M+N+1}\ot1\,.
$
Therefore \eqref{E2} is related by $\sim$ to the product
\begin{align*}
&
(\ts I_{\ts2}\ts\ldots\ts I_{\ts M+1}\,
J_{\ts M+2}\ts\ldots\ts J_{\ts M+N+1}\,
G_{\,2\ts\ldots\ts M+1}\,H_{\ts M+2\ts\ldots\ts M+N+1}\,
Q_{\,1,M+N+2}\ot1\ts)\ \times
\notag
\\[8pt]
&(\RS_{\,1,M+2}\ts(M\ts)\,\RS_{\,M+1,M+N+2}\ts(-N\ts)\,
G_{\,1\ts\ldots\ts M}\,H_{\ts M+3\ts\ldots\ts M+N+2}\ot1\ts)\ \times
\notag
\\[8pt]
&
T_1\ts(u+M-N)\,T_2\ts(u+M-N-1)\ts\ldots\ts T_{M}\ts(u-N+1)\,
\TS_{M+2}\ts(u-N)\ \times
\notag
\\[8pt]
&
T_{M+1}\ts(u-N)\,
\TS_{M+3}\ts(u-N+1)\ts\ldots\ts\TS_{M+N+1}\ts(u-1)\,\TS_{M+N+2}\ts(u)\ \times
\notag
\\[4pt]
&(\ts P_{\,1,M+1}\,P_{\,M+2,M+N+2}\ot1\ts)\,\frac1{M\ts!\,N\ts!}\,\ts. 
\label{E3}
\end{align*}
Due to the definition \eqref{rutinv} and to
Proposition \ref{L1} the latter product equals
\begin{align*}
&
(\ts P_{\,1,M+1}\,P_{\,M+2,M+N+2}\ot1\ts)\ \times
\\[10pt]
&
(\ts I_{\ts1}\ts\ldots\ts I_{\ts M}\,
J_{\ts M+3}\ts\ldots\ts J_{\ts M+N+2}\,
G_{\,1\ts\ldots\ts M}\,H_{\ts M+3\ts\ldots\ts M+N+2}\,\ts
Q_{\,M+1,M+2}\ot1\ts)\ \times
\\[8pt]
&
T_1\ts(u+M-N)\,T_2\ts(u+M-N-1)\ts\ldots\ts T_{M}\ts(u-N+1)\,
\TS_{M+2}\ts(u-N)\ \times
\\[8pt]
&
T_{M+1}\ts(u-N)\,
\TS_{M+3}\ts(u-N+1)\ts\ldots\ts\TS_{M+N+1}\ts(u-1)\,\TS_{M+N+2}\ts(u)\ \times
\\[6pt]
&(\ts P_{\,1,M+1}\,P_{\,M+2,M+N+2}\ot1\ts)\,\frac1{M\,N}\,=
(\ts P_{\,1,M+1}\,P_{\,M+2,M+N+2}\ot1\ts)\ \times
\\[6pt]
&
(\ts I_{\ts1}\ts\ldots\ts I_{\ts M}\,
J_{\ts M+3}\ts\ldots\ts J_{\ts M+N+2}\,
G_{\,1\ts\ldots\ts M}\,H_{\ts M+3\ts\ldots\ts M+N+2}\,\ts
Q_{\,M+1,M+2}\ot1\ts)\ \times
\\[12pt]
&
T_1\ts(u+M-N)\,T_2\ts(u+M-N-1)\ts\ldots\ts T_{M}\ts(u-N+1)\ \times
\\[4pt]
&
\TS_{M+3}\ts(u-N+1)\ts\ldots\ts\TS_{M+N+1}\ts(u-1)\,\TS_{M+N+2}\ts(u)\,
(\ts P_{\,1,M+1}\,P_{\,M+2,M+N+2}\ot1\ts)\,\frac1{M\,N}\,\ts.
\end{align*}
Here we have used 
\eqref{QTT}. Since $G$ and $H$ antisymmetrize the subspaces \eqref{SS},
by applying our map \eqref{STR} to the right hand side
of this equality and using the definition of $B(u+1)$ we get 
\begin{equation}
\label{BQ}
(-1)^N\,(M-1)\ts!\,(N-1)\ts!\,Q\ot B(u+1)\,.
\end{equation} 

Let us now consider the product \eqref{E4} which 
is equal to \eqref{E2} due to \eqref{E1}.
By again using Proposition \ref{P} 
the product \eqref{E4} can be rewritten as
\begin{align*}
&(\ts I_{\ts2}\ts\ldots\ts I_{\ts M+1}\,
J_{\ts M+2}\ts\ldots\ts J_{\ts M+N+1}\,
G_{\,2\ts\ldots\ts M}^{\,(M-1)}\,H_{\ts M+3\ts\ldots\ts M+N+1}^{\,(N-1)}\,
Q_{\,1,M+N+2}\ot1\ts)\ \times
\\[8pt]
&T_2\ts(u+M-N-1)\ts\ldots\ts T_{M}\ts(u-N+1)\,
\TS_{M+2}\ts(u-N)\,T_1\ts(u+M-N)\ \times
\\[8pt]
&\TS_{M+N+2}\ts(u)\,T_{M+1}\ts(u-N)\,
\TS_{M+3}\ts(u-N+1)\ts\ldots\ts\TS_{M+N+1}\ts(u-1)\ \times
\\[8pt]
&(\ts\RS_{\,1,M+2}\ts(M\ts)\,G_{\,2\ts\ldots\ts M}^{\,(M-1)}\,
R_{\,1M}\ts(M-1)\ts\ldots\ts R_{\,12}\ts(1)\ot1\ts)\ \times
\\[8pt]
&(\ts\RS_{\,M+1,M+N+2}\ts(-N\ts)\,H_{\ts M+3\ts\ldots\ts M+N+1}^{\,(N-1)}\,
R_{\,M+3,M+N+2}\ts(1-N)\ts\ldots\ts R_{\,M+N+1,M+N+2}\ts(-1)\ot1\ts)\ \times
\\[6pt]
&(\ts P_{\,1,M+1}\,P_{\,M+2,M+N+2}\ot1\ts)\,
\frac1{(M-1)\ts!\,(N-1)\ts!}\,\ts. 
\end{align*}

By again using \eqref{GM} and \eqref{HN} the latter product equals
\begin{align*}
&(\ts I_{\ts2}\ts\ldots\ts I_{\ts M+1}\,
J_{\ts M+2}\ts\ldots\ts J_{\ts M+N+1}\,
G_{\,2\ts\ldots\ts M}^{\,(M-1)}\,H_{\ts M+3\ts\ldots\ts M+N+1}^{\,(N-1)}\,
Q_{\,1,M+N+2}\ot1\ts)\ \times
\\[8pt]
&T_2\ts(u+M-N-1)\ts\ldots\ts T_{M}\ts(u-N+1)\,
\TS_{M+2}\ts(u-N)\,T_1\ts(u+M-N)\ \times
\\[8pt]
&\TS_{M+N+2}\ts(u)\,T_{M+1}\ts(u-N)\,
\TS_{M+3}\ts(u-N+1)\ts\ldots\ts\TS_{M+N+1}\ts(u-1)\ \times
\\[8pt]
&(\ts\RS_{\,1,M+2}\ts(M\ts)\,\RS_{\,M+1,M+N+2}\ts(-N\ts)\,
G_{\,1\ts\ldots\ts M}\,H_{\ts M+3\ts\ldots\ts M+N+2}\ot1\ts)\ \times
\\[4pt]
&(\ts P_{\,1,M+1}\,P_{\,M+2,M+N+2}\ot1\ts)\,
\frac1{(M-1)\ts!\,(N-1)\ts!}\,= 
\\[4pt]
&(\ts I_{\ts2}\ts\ldots\ts I_{\ts M+1}\,
J_{\ts M+2}\ts\ldots\ts J_{\ts M+N+1}\,
G_{\,2\ts\ldots\ts M}^{\,(M-1)}\,H_{\ts M+3\ts\ldots\ts M+N+1}^{\,(N-1)}\,
Q_{\,1,M+N+2}\ot1\ts)\ \times
\\[8pt]
&T_2\ts(u+M-N-1)\ts\ldots\ts T_{M}\ts(u-N+1)\,
\TS_{M+2}\ts(u-N)\,T_1\ts(u+M-N)\ \times
\\[8pt]
&\TS_{M+N+2}\ts(u)\,T_{M+1}\ts(u-N)\,
\TS_{M+3}\ts(u-N+1)\ts\ldots\ts\TS_{M+N+1}\ts(u-1)\ \times
\\[8pt]
&(\ts\RS_{\,1,M+2}\ts(M\ts)\,\RS_{\,M+1,M+N+2}\ts(-N\ts)\ot1\ts)\ \times
\\[4pt]
&(\ts P_{\,1,M+1}\,P_{\,M+2,M+N+2}\,
G_{\,2\ts\ldots\ts M+1}\,H_{\ts M+2\ts\ldots\ts M+N+1}
\ot1\ts)\,
\frac1{(M-1)\ts!\,(N-1)\ts!}\,\sim
\\[4pt]
&(\ts I_{\ts2}\ts\ldots\ts I_{\ts M+1}\,
J_{\ts M+2}\ts\ldots\ts J_{\ts M+N+1}\,
G_{\,2\ts\ldots\ts M+1}\,H_{\ts M+2\ts\ldots\ts M+N+1}\,Q_{\,1,M+N+2}
\ot1\ts)\ \times
\\[8pt]
&T_2\ts(u+M-N-1)\ts\ldots\ts T_{M}\ts(u-N+1)\,
\TS_{M+2}\ts(u-N)\,T_1\ts(u+M-N)\ \times
\\[8pt]
&\TS_{M+N+2}\ts(u)\,T_{M+1}\ts(u-N)\,
\TS_{M+3}\ts(u-N+1)\ts\ldots\ts\TS_{M+N+1}\ts(u-1)\ \times
\\[8pt]
&(\ts\RS_{\,1,M+2}\ts(M\ts)\,\RS_{\,M+1,M+N+2}\ts(-N\ts)\,
P_{\,1,M+1}\,P_{\,M+2,M+N+2}\ot1\ts)\,=
\\[8pt]
&(\ts I_{\ts2}\ts\ldots\ts I_{\ts M+1}\,
J_{\ts M+2}\ts\ldots\ts J_{\ts M+N+1}\,
G_{\,2\ts\ldots\ts M+1}\,H_{\ts M+2\ts\ldots\ts M+N+1}\ot Z(u)\ts)\ \times
\\[8pt]
&T_2\ts(u+M-N-1)\ts\ldots\ts T_{M}\ts(u-N+1)\,
\TS_{M+2}\ts(u-N)\ \times
\\[8pt]
&T_{M+1}\ts(u-N)\,
\TS_{M+3}\ts(u-N+1)\ts\ldots\ts\TS_{M+N+1}\ts(u-1)\ \times
\\[8pt]
&(\ts Q_{\,1,M+N+2}\,\RS_{\,1,M+2}\ts(M\ts)\,\RS_{\,M+1,M+N+2}\ts(-N\ts)\,
P_{\,1,M+1}\,P_{\,M+2,M+N+2}\ot1\ts)\,\ts.
\end{align*}
To obtain the last equality we used the definition of the series $Z(u)$
and the centrality in $\YMN$ of the coefficients of this serties, see the proof
of Proposition \ref{zu}.

Let us denote by $S(u)$ the product in the latter four displayed lines.
It is related by $\sim$ to the product \eqref{E4} which is equal to \eqref{E2}. 
We have already proved that 
the image of \eqref{E2} under our map \eqref{STR} is equal to \eqref{BQ}.
Hence the image of $S(u)$ is also equal to \eqref{BQ}.
In particular, the image of $S(u)$ under \eqref{STR}
does not change if we multiply this image on the right by
\begin{equation*}
(\ts I\ot1+1\ot J\ts)\ot1\in
(\End\CMN)^{\ot\ts2}\ot\YMN\,,
\end{equation*}
see \eqref{IJ}. Equivalently, the product $S(u)$
is related by $\sim$ to itself multiplied
on the right by
\begin{equation}
\label{IJ1}
(\ts I_{\ts1}+J_{\ts M+N+2}\ts)\ot1\in
(\End\CMN)^{\ot\ts(M+N+2)}\ot\YMN\,.
\end{equation}

\vbox{
Let us now right multiply $S(u)$ by \eqref{IJ1} and also by the element
\begin{equation}
\label{IJ2}
I_{\ts2}\ts\ldots\ts I_{\ts M+1}\,
J_{\ts M+2}\ts\ldots\ts J_{\ts M+N+1}\ot1\in
(\End\CMN)^{\ot\ts(M+N+2)}\ot\YMN\,.
\end{equation}
The result is still related by $\sim$ to $S(u)$ because \eqref{IJ2}
is a projector dividing $S(u)$ on the left. 
However, by multiplying the last line of $S(u)$ 
by both \eqref{IJ1} and \eqref{IJ2} we obtain the product
\begin{align*}
&(\ts Q_{\,1,M+N+2}\,\RS_{\,1,M+2}\ts(M\ts)\,\RS_{\,M+1,M+N+2}\ts(-N\ts)\,
P_{\,1,M+1}\,P_{\,M+2,M+N+2}\ot1\ts)\ \times
\\[8pt]
&(\ts(\ts I_{\ts1}+J_{\ts M+N+2}\ts)\,
I_{\ts2}\ts\ldots\ts I_{\ts M+1}\,
J_{\ts M+2}\ts\ldots\ts J_{\ts M+N+1}\ot1\ts)\,=
\\[8pt]
&(\ts Q_{\,1,M+N+2}\,\RS_{\,1,M+2}\ts(M\ts)\,\RS_{\,M+1,M+N+2}\ts(-N\ts)
\ot1\ts)\ \times
\\[8pt]
&
(\ts I_{\ts1}\ts\ldots\ts I_{\ts M}\,(\ts I_{\ts M+1}+J_{\ts M+2}\ts)\,
J_{\ts M+3}\ts\ldots\ts J_{\ts M+N+2}\,
P_{\,1,M+1}\,P_{\,M+2,M+N+2}\ot1\ts)\,.
\end{align*}Due to the definition \eqref{rutinv} and to
Proposition \ref{L2} the latter product equals
\begin{align*}
&(\ts I_{\ts2}\ts\ldots\ts I_{\ts M+1}\,
J_{\ts M+2}\ts\ldots\ts J_{\ts M+N+1}\ot1\ts)\ \times
\\[6pt]
&(\ts Q_{\,1,M+N+2}\,
\Bigl(-\,\frac1M\,P_{\,1,M+1}\,J_{\ts M+N+2}+
\frac1N\,P_{\,M+2,M+N+2}\,I_{\ts1}\ts\Bigr)
\,P_{\,1,M+1}\,P_{\,M+2,M+N+2}\ot1\ts)\ =
\\[8pt]
&(\ts I_{\ts2}\ts\ldots\ts I_{\ts M+1}\,
J_{\ts M+2}\ts\ldots\ts J_{\ts M+N+1}\ot1\ts)\ \times
\\[6pt]
&(\ts Q_{\,1,M+N+2}\,
\Bigl(-\,\frac1M\,P_{\,M+2,M+N+2}\,J_{\ts M+2}+
\frac1N\,P_{\,1,M+1}\,I_{\ts M+1}\ts\Bigr)\ot1\ts)\,\ts.
\end{align*}
Therefore $S(u)$ is related by $\sim$ to
\begin{align}
&(\ts I_{\ts2}\ts\ldots\ts I_{\ts M+1}\,
J_{\ts M+2}\ts\ldots\ts J_{\ts M+N+1}\,
G_{\,2\ts\ldots\ts M+1}\,H_{\ts M+2\ts\ldots\ts M+N+1}\ot Z(u)\ts)\ \times
\notag
\\[8pt]
&T_2\ts(u+M-N-1)\ts\ldots\ts T_{M}\ts(u-N+1)\,
\TS_{M+2}\ts(u-N)\ \times
\notag
\\[8pt]
&T_{M+1}\ts(u-N)\,
\TS_{M+3}\ts(u-N+1)\ts\ldots\ts\TS_{M+N+1}\ts(u-1)\ \times
\notag
\\[8pt]
&(\ts I_{\ts2}\ts\ldots\ts I_{\ts M+1}\,
J_{\ts M+2}\ts\ldots\ts J_{\ts M+N+1}\ot1\ts)\ \times
\notag
\\[4pt]
&(\ts Q_{\,1,M+N+2}\,
\Bigl(-\,\frac1M\,P_{\,M+2,M+N+2}\,J_{\ts M+2}+
\frac1N\,P_{\,1,M+1}\,I_{\ts M+1}\ts\Bigr)\ot1\ts)\,=
\notag
\\[6pt]
&(\ts I_{\ts2}\ts\ldots\ts I_{\ts M+1}\,
J_{\ts M+2}\ts\ldots\ts J_{\ts M+N+1}\,
G_{\,2\ts\ldots\ts M+1}\,H_{\ts M+2\ts\ldots\ts M+N+1}\ot 
Z(u)\,B(u)\ts)\ \times
\notag
\\[10pt]
&(\ts I_{\ts2}\ts\ldots\ts I_{\ts M+1}\,
J_{\ts M+2}\ts\ldots\ts J_{\ts M+N+1}\ot1\ts)\ \times
\notag
\\[4pt]
&(\ts Q_{\,1,M+N+2}\,
\Bigl(-\,\frac1M\,P_{\,M+2,M+N+2}\,J_{\ts M+2}+
\frac1N\,P_{\,1,M+1}\,I_{\ts M+1}\ts\Bigr)\ot1\ts)\,\sim
\notag
\\[6pt]
&(\ts I_{\ts2}\ts\ldots\ts I_{\ts M+1}\,
J_{\ts M+2}\ts\ldots\ts J_{\ts M+N+1}\,
G_{\,2\ts\ldots\ts M+1}\,H_{\ts M+2\ts\ldots\ts M+N+1}\ot 
Z(u)\,B(u)\ts)\ \times
\notag
\\[4pt]
&(\ts Q_{\,1,M+N+2}\,
\Bigl(-\,\frac1M\,P_{\,M+2,M+N+2}+
\frac1N\,P_{\,1,M+1}\ts\Bigr)\ot1\ts)\,.
\label{final}
\end{align}
To obtain the equality in \eqref{final} we also used the relation 
\eqref{referee2} which in this instance implies~that
\begin{align*}
&(\ts I_{\ts M+1}\,J_{\ts M+2}\ot1\ts)\,
\TS_{M+2}\ts(u-N)\,T_{M+1}\ts(u-N)\,
(\ts I_{\ts M+1}\,J_{\ts M+2}\ot1\ts)=
\\[6pt]
&(\ts I_{\ts M+1}\,J_{\ts M+2}\ot1\ts)\,
T_{M+1}\ts(u-N)\,\TS_{M+2}\ts(u-N)\,
(\ts I_{\ts M+1}\,J_{\ts M+2}\ot1\ts)\,.
\end{align*}
}

\newpage

We could now show by direct calculation that applying the map \eqref{STR}
to the product in the last two lines of \eqref{final} yields
\begin{equation}
\label{BQZ}
(-1)^N\,(M-1)\ts!\,(N-1)\ts!\,Q\ot Z(u)\,B(u)\,.
\end{equation} 
Theorem \ref{BZ} would then 
follow because the equality of 
\eqref{E2} and \eqref{E4} implies the equality of \eqref{BQ} and \eqref{BQZ}.
However, we will complete the proof of Theorem~\ref{BZ} by an indirect
argument. We have already proved that the image of the product in
the last two lines of \eqref{final} equals~\eqref{BQ}.
Since the image of the action of $Q$ on $(\CMN)^{\ts\ot\ts2}$
is one dimensional, the latter
equality implies that $Z(u)\,B(u)$ equals $B(u+1)$ up to a scalar factor.
This scalar factor is~$1$ because the leading terms of both
series $B(u)\,Z(u)$ and $B(u+1)$ are $1\,$.
Theorem~\ref{BZ} is now proved.
 

\pdfbookmark[1]{References}{ref}

\LastPageEnding

\end{document}